\documentclass[final,12pt]{colt2026final} 

\usepackage[usenames,dvipsnames]{xcolor}
\usepackage{fancyhdr}
\usepackage{amsfonts,amsbsy,amsgen,amscd,mathrsfs}

\usepackage{stackengine}
\usepackage{pifont}
\usepackage{mathtools}
\usepackage{tcolorbox}
\newcommand{\MSE}{\text{MSE}}

\DeclareMathOperator{\E}{\mathbb{E}}
\newcommand{\ind}[1]{\mathbf{1}_{\{#1\}}}

\newcommand{\LLT}{\text{LLT}}

\newcommand{\Z}{\mathbb{Z}}
\newcommand{\A}{\mathcal A}
\newcommand{\D}{\mathcal D}
\newcommand{\Cov}{\operatorname{Cov}}

\newtheorem{questioncounter}{Question}
\newtheorem{question}[questioncounter]{Question}

\definecolor[named]{myLipicsLightGray}{rgb}{0.85,0.85,0.86}

\newenvironment{ShadedBox}{\begin{tcolorbox}[colback=myLipicsLightGray,colframe=myLipicsLightGray,arc=-1pt]\begin{minipage}{0.975\textwidth}}{\end{minipage}\end{tcolorbox}}

\graphicspath{{figures/}}

\usepackage{tikz}
\usetikzlibrary{arrows,chains,matrix,positioning,scopes}
\usepackage{tikz-3dplot}

\usepackage{xcolor}

\definecolor{cit}{rgb}{0.91,0.39,0.16}

\definecolor{dark-gray}{gray}{0.3}

\definecolor{dkgray}{rgb}{.3,.3,.3}
\definecolor{medgray}{rgb}{.5,.5,.5}
\definecolor{ltgray}{rgb}{.7,.7,.7}
\definecolor{dkblue}{rgb}{0,0,.5}
\definecolor{medblue}{rgb}{0,0,.75}
\definecolor{ltblue}{rgb}{0.97,0.97,1}
\definecolor{rust}{rgb}{0.5,0.1,0.1}
\definecolor{ltyellow}{rgb}{1, 1, 0.9}

\usepackage{soul}
\sethlcolor{ltyellow}

\usepackage[normalem]{ulem}

\usepackage{hyperref}

\hypersetup{hidelinks,colorlinks=true,breaklinks=true,urlcolor=cit,citecolor=dkblue,linkcolor=dkblue,bookmarksopen=false,hypertexnames=false}

\usepackage{bookmark}
\bookmarksetup{
open,
numbered,
addtohook={%
\ifnum\bookmarkget{level}=0 %
\bookmarksetup{bold}%
\fi
\ifnum\bookmarkget{level}=-1 %
\bookmarksetup{color=cit,bold}%
\fi
}
}

\usepackage[full]{textcomp}

\usepackage[scaled=.98,sups,osf]{XCharter}%
\usepackage[scaled=1.04,varqu,varl]{inconsolata}%
\usepackage[type1,scaled=0.98]{cabin}%
\usepackage[utopia,vvarbb,scaled=1.07]{newtxmath}
\usepackage[cal=euler]{mathalfa}
\usepackage{bm} %

 \usepackage{microtype} %
 \usepackage[utf8]{inputenc} %
 \usepackage[T1]{fontenc} %
 \usepackage{setspace}
\usepackage[capitalize]{cleveref}
\usepackage{enumitem}

\setlist{noitemsep} %

\title[Minimax Limits of \(k\)-Fold Cross-Validation]{
Minimax Limits of \(k\)-Fold Cross-Validation via Majority
}
\usepackage{times}

\coltauthor{
 \Name{Ido Nachum} \Email{inachum@univ.haifa.ac.il}\\
 \addr University of Haifa, Haifa, Israel.
 \AND\Name{Rüdiger Urbanke}\Email{rudiger.urbanke@epfl.ch}\\
 \addr EPFL, Lausanne, Switzerland.
 \AND
 \Name{Thomas Weinberger}\Email{thomas.weinberger@epfl.ch}\\
 \addr EPFL, Lausanne, Switzerland.}

\begin{document}

\maketitle

\begin{abstract}
We study the mean-squared error of \(k\)-fold cross-validation as a risk estimator, with particular emphasis on how its accuracy depends on the number of folds \(k\). Despite the widespread use of cross-validation, principled guidance for choosing \(k\) is largely absent, mainly due to the complex dependence between fold-wise error estimates. To obtain sharp and interpretable results, we focus on the majority algorithm in binary classification, a minimal yet nontrivial empirical risk minimization procedure. We provide a fine-grained analysis of its cross-validation behavior, showing that even this simple algorithm exhibits subtle and delicate phenomena for which existing theory provides loose and even vacuous bounds. Leveraging this analysis, we introduce a minimax framework for cross-validation risk estimation and prove that no empirical risk minimization algorithm can achieve an \(O(1/n)\) minimax mean-squared error when the number of folds grows with the number of samples \(n\); instead, a lower bound of order \(\Omega(\sqrt{k}/n)\) is unavoidable. Our results reveal fundamental limitations of cross-validation as a data-reuse strategy, clarify gaps and inaccuracies in prior theoretical work, and position the majority algorithm as a natural benchmark that any tight analysis of cross-validation should be able to explain.
\end{abstract}
\begin{keywords}%
 cross-validation, learning theory, algorithmic stability
\end{keywords}

\section{Introduction}

$k$-fold cross-validation (CV) is a widely used model validation technique in statistics, data science, and machine learning; see \citep{arlot2010survey} for a comprehensive survey. Given a collection of errors obtained by training models on subsets of the data and validating them on the remaining observations, CV is typically used for one of two purposes:  
\emph{(risk estimation)} estimating the risk of a model chosen independently of the error computations by averaging the validation errors; or  
\emph{(model selection)} selecting, among a collection of candidate models, the one that minimizes the CV error.

The primary motivation for using $k$-fold CV for risk estimation is that partitioning the data into multiple non-overlapping subsets generally reduces statistical variability compared to relying on a single hold-out set \citep{blum}. Moreover, unlike the empirical (training) error, CV typically mitigates the overly optimistic bias caused by overfitting, a phenomenon that is particularly pronounced for overparametrized models such as neural networks.

Many CV schemes have been proposed, including combinatorial partitioning \citep{combinatCV} and Monte Carlo resampling \citep{montecarlo}. In this work, we focus on the standard variant in which the sample is partitioned into $k$ non-overlapping folds of (approximately) equal size; $5$-fold and $10$-fold CV are especially common in practice. Despite its pervasive use as a validation tool in the empirical sciences, the theoretical properties of CV remain surprisingly poorly understood. In particular, there is still no principled method for choosing the number of folds $k$. As noted by \citet[Sec.~10.3]{arlot2010survey}:

\begin{center}
\begin{quote}
\textit{``VFCV [V-fold cross-validation] is certainly the most popular CV procedure, in particular because of its mild computational cost. Nevertheless, the question of choosing V remains widely open, even if indications can be given towards an appropriate choice.''}
\end{quote}
\end{center}
A central obstacle is that fold-wise errors are dependent, and this dependence is difficult to characterize sharply. As a result, relatively few works provide a precise mathematical analysis of how CV-based risk-estimation accuracy varies with $k$;  see, for example, \citep{specific1,specific2}, which identify optimal choices of $k$ in specific density estimation settings.

To ground the discussion, consider first an extreme baseline: binary classification with the $0$--$1$ loss and a constant learning algorithm $\mathcal{A}_h$ that always outputs the same hypothesis $h$, independent of the sample. If the population risk of $h$ is $p$, then the mean-squared error (MSE) of the $k$-fold CV risk estimator satisfies
\(
\mathrm{MSE}_{\mathrm{CV}}^{(k)}=\frac{p(1-p)}{n}
\)
independent of $k$ since $L(\mathcal{A}_h,\mathcal{D})=p$ and $n\cdot\widehat L_{\mathrm{CV}}^{(k)} \sim \mathrm{Bin}(n,p)$, so $\mathrm{MSE}_{\mathrm{CV}}^{(k)}=\mathrm{Var}(\mathrm{Bin}(n,p))/n^2$.

A natural next step is to consider the simplest nontrivial extension of this baseline: the majority algorithm. Let $S^n=\{(x_i,y_i)\}_{i=1}^n\sim\mathcal{D}^n$ be an i.i.d.\ sample, and define $Y:=\sum_{i=1}^n y_i$. The majority algorithm is
\[
\mathcal{A}_{\mathrm{maj}}(S^n)=
\begin{cases}
h_0:x\mapsto 0, & \text{if } Y\le n/2,\\[4pt]
h_1:x\mapsto 1, & \text{if } Y> n/2,
\end{cases}
\]
where $h_i$ denotes the constant hypothesis outputting $i$.

This is arguably the simplest empirical risk minimization (ERM) extension of the constant predictor: the hypothesis class grows from one constant function to two, and the decision rule depends only on the label counts (ignoring the features entirely). Yet, a precise analysis of its CV behavior is surprisingly delicate. Our study has two components:

\begin{enumerate}
\item $k$-fold CV is explicitly designed to reuse the same data for both training and validation. Accordingly, our aim is to isolate the effect of this data reuse on CV-based risk estimation.
To this end, we benchmark CV-based risk estimation against an idealized baseline in which one trains on all $n$ available samples and, in addition, has access to an independent validation set of size $n$.
\item We conduct a fine-grained analysis of the majority algorithm, which allows us to identify and rigorously analyze several gaps in the existing theoretical literature, including overly loose bounds.
\end{enumerate}

\subsection{Our Contributions}

\paragraph{Fine-grained analysis of Majority.}

In Section~\ref{sec:maj}, we carry out a fine-grained analysis of $\mathrm{MSE}_{\mathrm{CV}}^{(k)}$ for the Majority algorithm under a random-label model. The dominant contribution to the MSE arises from the covariance between validation errors across folds (see the notation in Section~\ref{sec:notation}). In general, this covariance is difficult to analyze due to the complex dependencies induced by data reuse across folds.

We identify a structural property that the Majority algorithm satisfies that allows these dependencies to be handled explicitly. This property, which we call \emph{Factorization} (Lemma~\ref{lem:factor}), reduces the covariance computation to a simpler quantity depending only on the distribution of the sample within a single fold. Leveraging this result, we derive a closed-form expression for the covariance in Theorem~\ref{thm:exactComb}, and hence for $\mathrm{MSE}_{\mathrm{CV}}^{(k)}$. In particular, we show in Theorem~\ref{thm:majority} that  $\mathrm{MSE}_{\mathrm{CV}}^{(k)}$ scales as $\Theta({\sqrt{k}/n})$.

\paragraph{Minimax limits of cross-validation.} 
\citet{kearns} proved that for loss-stable ERMs over VC classes, leave-one-out CV performs essentially no worse than the empirical error. Similar results for general $k$-fold CV were obtained by \citet{anthony}.

Another natural requirement is that CV performs no worse than a hold-out estimate over a single fold of size $n/k$. The work of \citet{blum} confirms this property for a specific (non-standard) cross-validation setting, where the algorithm’s final output on the full sample $S^n$ is the random classifier sampling uniformly at random from the $k$ cross-validated hypotheses.

Importantly, these works do not \emph{quantify} the advantage of CV. To make this notion precise, we define the \emph{minimax cross-validation risk} for a given algorithm \(\mathcal{A}\) as
\[
\mathfrak{R}_{\mathrm{CV}}(\mathcal{A})
:= \min_{k\mid n}\max_{\mathcal{D}} 
   \mathrm{MSE}_{\mathrm{CV}}^{(k)}(\mathcal{A},\mathcal{D}),
\]
which represents the optimal achievable MSE over all choices of $k$ in the absence of knowledge about the underlying distribution $\mathcal{D}$.

 The constant algorithm discussed earlier attains its minimax for uniformly random labels, hence \(\mathfrak{R}_{\mathrm{CV}}(\mathcal{A}_h) =1/(4n)\). This provides a natural baseline or reference point to compare more sophisticated algorithms. Since overly simplistic procedures such as the constant algorithm are not practically useful—yet may achieve similar minimax rates—we focus on (non-constant) ERM algorithms, a standard assumption in statistical learning theory. This leads to the following question:

\begin{ShadedBox}
  \begin{question}\label{q2}
    In classification with 0--1 loss, can any ERM achieve an \(O(1/n)\) minimax rate for CV risk estimation, and if not, how close can it get?
  \end{question}
\end{ShadedBox}

In Section~\ref{sec:minimax_maj}, we show that ERM algorithms cannot achieve a minimax rate of \(O(1/n)\) for $k=\omega_n(1)$ using a reduction to our in-depth study of the Majority algorithm:
For any ERM algorithm and for every $k$ it holds

$$ \max_{\mathcal{D}} 
   \mathrm{MSE}_{\mathrm{CV}}^{(k)}(\mathcal{A},\mathcal{D}) = \Omega( \sqrt{k}/n ).$$ 
Therefore, for any ERM algorithm \(\mathcal{A}\) that achieves the minimax optimum with  \(k^*\) folds, the MSE of cross-validation scales as
\[
\mathfrak{R}_{\mathrm{CV}}(\mathcal{A}) = \Omega\!\left(\sqrt{k^*}/n \right).
\]
Hence, even for the most carefully designed ERM, CV cannot  exploit the entire dataset as if it were using a single hold-out set of size \(n\); there remains a factor of \(\sqrt{k^*}\) in the rate.

\paragraph{Limitations of existing theory and the Majority benchmark.}

A substantial body of work, most notably
\citep{rogers,devroye,blum,kearns,bousquet,kumar2013near}, has yielded important insights into the behavior of CV and significantly shaped the theoretical landscape. Nevertheless, when these results are evaluated against the Majority algorithm, they exhibit systematic shortcomings. In particular, existing guarantees typically suffer from at least one of the following limitations:

\begin{enumerate}

  \item \emph{Arbitrarily loose sufficient conditions.}
Stability-based upper bounds, including those of \citet{kumar2013near}, fail to capture the performance of even the majority algorithm. While the proof techniques underlying the results of  \citet{rogers,devroye,kearns,bousquet} naturally extend to general $k$-fold cross-validation from leave-one-out CV, the resulting guarantees remain fundamentally misaligned with the behavior of Majority.

  \item \emph{Relative rather than absolute guarantees.}
  Several works provide guarantees only in comparison to other error measures rather than in absolute terms. For instance, \citet{kearns} compare leave-one-out CV to empirical error; however, our analysis of Majority shows that these bounds are loose across most regimes of their confidence parameter $\delta$, particularly when considering the MSE of CV loss estimation. 

  \item \emph{Mathematical inaccuracies in prior analyses.}
  Independent of our analysis of Majority, we identify errors in fundamental theorems in \cite{kearns} and \cite{kale2011cross}. Clarifying these issues is important because the affected results make claims about fundamental properties of CV estimation, and correcting them reveals several unresolved theoretical questions.
\end{enumerate}


We discuss the Majority benchmark in Section~\ref{sec:lim} and the mathematical inaccuracies in
Section~\ref{sec:errors}. The benchmark analysis leads to the following conclusion:

\begin{ShadedBox}
 Majority is a natural benchmark and we advocate that demonstrating tightness for this instance should be a minimal requirement for any future bounds on the error of CV.
\end{ShadedBox}

\section{Setup and Notation}\label{sec:notation}

We start by establishing the framework for our investigation. Let \(\mathcal{X}\) be the input space and \(\mathcal{Y}\) the output space, and set \(\mathcal{Z}=\mathcal{X}\times\mathcal{Y}\).
We study (possibly randomized) learning rules \(\mathcal{A}:\mathcal{Z}^\star\to\mathcal{H}\) that map a sample
\(S^n=(Z_1,\dots,Z_n)\in\mathcal{Z}^n\) to a hypothesis \(h=\mathcal{A}(S^n)\in\mathcal{H}\subseteq \mathcal{Y}^{\mathcal{X}}\).
The observations are i.i.d.: \(Z_i\sim\mathcal{D}\), hence \(S^n\sim\mathcal{D}^{ n}\).
The learning algorithm \(\mathcal{A}\) may be randomized; in that case, $\mathbb E_\mathcal A[\cdot]$ denotes expectation with respect to the internal randomness of $\mathcal A$.
As is common in previous works, we assume throughout that \(\mathcal{A}\) is permutation-invariant (symmetric): for any permutation
\(\pi\) of \(\{1,\dots,n\}\),
\(
\mathcal{A}(S^n)=\mathcal{A}(S^n_\pi)\) a.s., where
\(S^n_\pi:=(Z_{\pi(1)},\dots,Z_{\pi(n)})\).

We assume throughout that the number of folds is an integer \(k \ge 2\) with \(k\mid n\).
We partition the index set \(\{1,\dots,n\}\) into \(k\) disjoint blocks \(I_1,\dots,I_k\) of size \(m:=|I_i|=n/k\),
and define \(S_i=\{Z_j:j\in I_i\}\) and \(S_{-i}=S^n\setminus S_i\). We denote the concatenation of two samples $S_i$ and $S_j$ by $S_i \circ S_j$.

Given a loss function \(\ell:\mathcal{Y}\times\mathcal{Y}\to\mathbb{R}\), the \(k\)-fold cross-validation estimator is 
\[
\widehat L_{\mathrm{CV}}^{(k)}(\mathcal A,S^n)=\frac{1}{k}\sum_{i=1}^{k}\widehat L_i^{(k)},
\quad
\widehat L_i^{(k)}=\frac{1}{|S_i|}\sum_{(x,y)\in S_i}\ell\big(\mathcal{A}(S_{-i})(x),\,y\big)
=\frac{k}{n}\sum_{(x,y)\in S_i}\ell\big(\mathcal{A}(S_{-i})(x),\,y\big).
\]
That is, \(\widehat L_i^{(k)}\) is the average loss on the \(i\)th hold-out fold, and
\(\widehat L_{\mathrm{CV}}^{(k)}\) averages these across folds.\footnote{When $k=n$, the definition coincides with leave-one-out cross-validation.} We omit the subscript $\mathrm{CV}$  whenever it is clear from context. Since this work focuses on classification, we will often consider the 0--1 loss defined as $\ell(\hat y,y)=\mathbb 1_{\{\hat y\neq y\}}$.

We assess the performance of cross–validation via the mean-squared error (MSE)
\[
\mathrm{MSE}_{\mathrm{CV}}^{(k)}(\mathcal{A}, \mathcal{D})
:= \mathbb{E}_{\,S^n\sim \mathcal{D}^{ n},\,\mathcal A}
\!\left[\big(   \widehat L_{\mathrm{CV}}^{(k)}(\mathcal A,S^n)   -L(\mathcal A(S^n) )\big)^2\right],
\]
where the population risk is
\(
L(h):=\mathbb{E}_{(x,y)\sim\mathcal D}\big[\ell(h(x),y)\big].
\)

For the \(i\)th fold, we also write
\(
L_i^{(k)}(S^n):=   L(\mathcal A(S_{-i})),
\)
i.e., the risk of the hypothesis trained on the complement \(S_{-i}\) of the \(i\)th hold–out block.

\section{Preliminaries on Algorithmic Stability}
Before positioning our work within the context of previous works, it is instructive to familiarize oneself with commonly used notions of algorithmic stability. While there are many notions of algorithmic stability in the literature, we will focus on two of the most widely used variants. We also note that most classical works on the performance of leave-one-out CV consider the following notions for the special case where $m=1$, while some newer works also consider leave-$m$ notions with $m>1$ \citep{gastpar2024algorithms}.

\begin{definition}[Hypothesis Stability]\label{def:hyp_stab}
We call a pair $(\mathcal A,\mathcal D)$ \textit{hypothesis-stable} with respect to some metric $\mathrm{dist}(\cdot \,, \cdot)$ with parameters $(\beta,m)$ if 
\[
\mathbb E_{S^{n-m}\sim \mathcal D^{n-m}, ~S^{m} \sim \mathcal D^{m}, ~\mathcal A} \Big[\mathrm{dist}\big(\mathcal A(S^{n-m}\circ S^{m}) , \mathcal A(S^{n-m})\big)\Big] \le \beta.
\]
\end{definition}

Intuitively, hypothesis stability is a stronger assumption than necessary for error estimation. It provides a quantitative measure of how similar the hypotheses trained on different folds are to the one obtained from the full dataset. In this sense, a highly hypothesis-stable algorithm behaves almost like a constant algorithm. However, the key factor governing the accuracy of cross-validation (CV) error estimation is not the similarity of hypotheses themselves, but rather the stability of their loss values when a small subset of training samples is removed.

For this reason, it is more natural to require a weaker property, called \emph{loss stability} (or \emph{error stability}). This condition ensures that the per-fold loss estimates remain nearly unbiased, even when the training data is slightly perturbed.
\begin{definition}[Loss Stability]\label{def:lossstab}
We call a pair $(\mathcal A,\mathcal D)$ \textit{loss-stable} with parameters $(\beta,m)$ if 
$$
\mathbb E_{S^{n-m}\sim \mathcal D^{n-m},~S^{m}\sim \mathcal D^{m},~\mathcal A} \Big[\big|L\big(\mathcal A(S^{n-m}\circ S^{m})\big)-L\big(\mathcal A(S^{n-m})\big)\big|\Big] \le \beta.
$$
\end{definition}
Loss stability may appear necessary for accurate risk estimation: each fold estimate is unbiased for the risk of the algorithm trained on \(n-m\) samples, so CV should estimate \(L(\mathcal A(S^n))\) accurately only if this risk is close to \(L(\mathcal A(S^{n-m}))\). However, this reasoning is incomplete because fold estimates are correlated. In fact, pathological algorithms can have poor loss stability while still admitting perfectly accurate CV risk estimates.

\begin{lemma}\label{lem:anticorr}
In the setting of classification with the 0--1 loss, there exist algorithm-distribution combinations whose loss-stability is \(\beta=1/4\), yet whose CV MSE is $0$.
\end{lemma}
\begin{proof} See Appendix~\ref{proof:anti}.
\end{proof}
This result contradicts Theorem~5.3 of \citet{kearns}; as discussed in Appendix~\ref{app:err}, the discrepancy stems from an error in that theorem. In general, it remains unclear in which cases loss stability is strictly necessary for small MSE.

\section{Prior Work}
The works of \citet{rogers,devroye,devroye2} were among the first to establish rigorous stability-based performance guarantees for classification problems using leave-one-out CV. Although in their works, they assume that the considered algorithms are \emph{local} (e.g. nearest neighbors) and the data distributions are arbitrary, their results directly generalize to the class of hypothesis-stable algorithms (in which case the bounds are no longer distribution-free).

The well-known work by \citet{bousquet} provided a streamlined presentation of classical results and novel error bounds for leave-one-out CV and the empirical error under various strengthened assumptions on algorithmic stability and/or the loss functional.

Estimating the population loss in an algorithm-dependent manner is closely related to statistical learning theory (see, e.g., \citep{UML}). The principal aim of this field is the development of generalization bounds, typically in the form of high-probability upper bounds $L(\mathcal A(S))<\widehat L_{\text{emp}}(\mathcal A(S),S)+C$, where $\widehat L_{\text{emp}}(\mathcal A(S),S)$ denotes the empirical error over the training set and the generalization measure $C$ accounts for the over-optimism of the empirical error induced by the complexity of the model. 
A classical result \citep{vapnik,blumer} for binary classification with 0--1 loss states it is necessary and sufficient to let $C=\Theta (\sqrt {d/n})$ to ensure that the generalization bound holds in a tight manner even for the worst-case distribution, where $d$ is the VC dimension, a combinatorial measure of the richness of the hypothesis class $\mathcal H$ associated with the algorithm. 
These bounds are often too pessimistic because they are not sensitive to the (possibly benign) characteristics of the specific data distribution at hand. Moreover, it can be shown that in overparametrized settings (which are ubiquitous in machine learning), generalization measures that are not distribution-dependent face limitations both empirically \citep{jiangfantastic, dziugaite2020search} and theoretically \citep{nachum2024fantastic,gastpar2024algorithms}. 
With this in mind, CV becomes conceptually interesting as a flexible alternative to generalization measures for overparametrized settings, where the empirical error is typically uninformative, and a distribution-dependent measure is required---though admittedly CV is no silver bullet (theoretical bounds require estimating the algorithm's stability, and CV can be computationally expensive). 

In light of this comparison, a sound minimal requirement is that CV performs at least as well as the empirical error. This question, formalized in terms of so-called \textit{sanity-check bounds}, has been studied by \citet{kearns}. One of their central results is that, for loss-stable empirical risk minimizers over VC classes, leave-one-out CV is guaranteed to perform essentially no worse than the empirical error.
\cite{anthony} derived similar results for the more general case of $k$-fold CV.

Yet another natural sanity-check is to require CV to perform no worse than a single hold-out set of corresponding size. The work of \citet{blum} shows that this does indeed hold for a specific (non-standard) cross-validation setting.

An influential line of work \citep{bengioUnbiased,bengio2} considers the limitations of unbiased estimation of the variance of CV.

A more recent line of work develops upper bounds on the MSE based on novel notions of loss stability \citep{kale2011cross, kumar2013near}.
As discussed in Appendix~\ref{app:errKale}, the main theorem of \citet{kale2011cross} contains a technical error, which makes it difficult to assess the implications of the stated results.
The follow-up work introduces a version of loss stability that leads to a stronger result \citep[Theorem 1]{kumar2013near} since the related stability parameter is a lower bound on the one appearing in \citep[Theorem 2]{kale2011cross}.
In both works, the authors aim to bound the performance of the non-standard algorithm that at test time picks one of the  cross-validated hypotheses uniformly at random, while we directly bound the MSE of the full-sample hypothesis.

\citet{bayle} develop a central limit theorem for $k$-CV error around the $k$-fold test error under an abstract asymptotic linearity condition and propose consistent variance estimators that yield asymptotically exact confidence intervals and hypothesis tests for comparing two learning algorithms.

\section{Results}\label{sec:results}

We first present an exact analysis of the majority algorithm in binary classification
with the 0--1 loss. Accordingly, the results in Sections~\ref{sec:maj}--\ref{sec:lim} are specific to
the binary setting $\mathcal Y=\{0,1\}$ and to
$\ell(\hat y,y)=\mathbb 1_{\{\hat y\neq y\}}$, with the exception of Lemma~\ref{lem:factor}. The subsequent minimax lower bound is
proved by reducing arbitrary ERM algorithms to this binary majority problem. Since
binary classification is a special case of multiclass classification, our lower bounds directly extend to multiclass classification with 0--1 loss.

\paragraph{Admissible fold sizes.} Throughout this section we work with equal-size folds: whenever a fold count \(k\) is considered, we assume \(k\mid n\) and write \(m=n/k\). This convention is only meant to suppress inessential rounding effects. If \(k\nmid n\), one could instead use approximately equal folds, or carry all arguments through with floor and ceiling corrections; for Majority, this only replaces the single exchangeable fold covariance by fold-size-dependent variants and does not change the covariance mechanism or its qualitative dependence on \(k\). Thus, fixed-\(k\) asymptotics are understood along sequences with \(k\mid n\); for example statements comparing \(k=2\) and \(k=3\) are understood along subsequences with \(6\mid n\).

\subsection{MSE of Cross-Validation for Majority}\label{sec:maj}

We analyze the majority algorithm under a distribution $\mathcal{D}$ whose marginal over $\mathcal{X}$ is arbitrary, and whose labels $y_i$ are i.i.d.\ draws from $\mathcal{Y} = \{0,1\}$ with $y_i \sim \mathrm{Ber}(1/2)$. 
In this case, $Y \sim \mathrm{Bin}(n, 1/2)$, and the population loss of $\mathcal{A}_{\mathrm{maj}}$ equals $1/2$, independent of both the sample $S$ and the sample size $n$. Consequently, analyzing its MSE reduces to controlling the covariance between folds.

\begin{proposition}\label{prop:majMSE}
The MSE of the majority algorithm with random  labels equals \[\frac{k-1}{k}\Cov(\widehat L_1^{(k)},\widehat L_2^{(k)}) + \frac{1}{4n}.\]
\end{proposition}
\begin{proof}
Since we assume the labels are randomly distributed, $L(\mathcal{A}_{\mathrm{maj}}(S^n)) = 1/2$ almost surely and the estimator is unbiased, i.e., $\mathbb{E}[\widehat L_{\mathrm{CV}}^{(k)}] = 1/2$. Hence the MSE reduces to the variance, which can be decomposed as
\(
    \mathrm{MSE}_{\mathrm{CV}}^{(k)}(\mathcal{A}_{\mathrm{maj}}, \mathcal{D}) = \mathrm{Var}\left( \frac{1}{k} \sum_{i=1}^k \widehat L_i \right)= \frac{1}{k} \mathrm{Var}(\widehat L_1) + \frac{k-1}{k} \mathrm{Cov}(\widehat L_1, \widehat L_2)
\)
where
$\mathrm{Var}(\widehat L_1) =\frac{1}{4m}$
since $\widehat L_1$ is independent of $S_{-1}$ and $m \cdot\widehat L_1 \sim \mathrm{Bin}(m, 1/2)$.
\end{proof}

\subsubsection{Fold-Covariance and Loss Factorization}
As we saw, for the simple class of algorithms with constant loss (such as Majority), the MSE reduces to the fold-covariance.
In such settings, the fundamental challenge in analyzing cross-validation is the complex dependency structure introduced by the overlapping training sets. For any two distinct folds $i$ and $j$, the training sets $S_{-i}$ and $S_{-j}$ share a large amount of data $S \setminus (S_i \circ S_j)$. This overlap creates a coupling between the cross-validated hypotheses $h_i=\A(S_{-i})$ and $h_j=\A(S_{-j})$ that is notoriously difficult to quantify.

We identify a structural property—\textit{loss factorization}—that allows us to disentangle these dependencies. If the empirical loss can be separated into a local \textit{test statistic} $\phi(S_i)$ and a global \textit{decision function} $\psi(S_{-i})$, the complex interaction collapses into a simpler form.

Formally, the factorization $\widehat{L}_i = \mu + \phi(S_i)\psi(S_{-i})$ models a class of algorithms where the interaction between the folds is separable: the validation fold $S_i$ influences the loss only through the scalar summary $\phi(S_i)$, and the training set $S_{-i}$ influences the loss only through the response $\psi(S_{-i})$.

\begin{lemma}[Factorization Lemma]\label{lem:factor}
    Let $S':=S\backslash (S_1 \circ S_2)$ denote the shared training set of the first two cross-validated hypotheses. Assume that the empirical losses of the cross-validated hypotheses have the form
	\[
		\widehat{L}_i = \mu +\phi(S_i)\psi(S_{-i}),
	\]
	where $\mu = \mathbb{E}[\widehat{L}_i]$. Then,
	\[
		\mathrm{Cov}(\widehat{L}_1, \widehat{L}_2) = \mathbb{E}_{S'} \left[ \left( \mathbb{E}_{S_1} [ \phi(S_1)\psi(S'\circ S_1) \mid S' ] \right)^2 \right].
	\]
    Moreover, if $\phi(S_1)$ is centered ($\mathbb{E}[\phi(S_1)] = 0$), then
	\[
		\mathrm{Cov}(\widehat{L}_1, \widehat{L}_2) = \mathbb{E}_{S'} \left[ \mathrm{Cov}(\phi(S_1), \psi(S'\circ S_1) \mid S')^2 \right].
	\]
\end{lemma}

\begin{proof}
	Due to the form of the empirical fold loss, the covariance is equal to the inner product
	\begin{align*}
		\mathrm{Cov}(\widehat{L}_1, \widehat{L}_2) &= \mathbb{E} \left[ \phi(S_1)\psi(S'\circ S_2) \cdot \phi(S_2)\psi(S'\circ S_1) \right]\\
        &=\mathbb{E} \left[ \phi(S_1)\psi(S'\circ S_1) \cdot \phi(S_2)\psi(S'\circ S_2) \right].
	\end{align*}
	Conditioning on $S'$, the terms $T_1:=\phi(S_1)\psi(S'\circ S_1)$ and $T_2:=\phi(S_2)\psi(S'\circ S_2)$ become independent and identically distributed (since $S_1 \stackrel{d}{=} S_2$ and they are disjoint from $S'$):
	\begin{align*}
		\mathbb{E}[T_1 T_2 \mid S'] = \mathbb{E}_{S_1} [ T_1 \mid S' ] \cdot \mathbb{E}_{S_2} [ T_2 \mid S' ] = \left( \mathbb{E}_{S_1} [ T_1 \mid S' ] \right)^2.
	\end{align*}
	Taking the expectation over $S'$ yields the result.

    \emph{Centered case:} by the definition of conditional covariance,
    \[
\mathbb E[\phi(S_1)\psi(S'\circ S_1)\mid S']=
\operatorname{Cov}(\phi(S_1),\psi(S'\circ S_1)\mid S')
+
\mathbb E[\phi(S_1)\mid S']
\mathbb E[\psi(S'\circ S_1)\mid S'].
\]
Since \(S_1\) is independent of \(S'\) and \(\mathbb E[\phi(S_1)]=0\), the second term vanishes.
\end{proof}

The term $\mathrm{Cov}(\phi(S_1), \psi(S' \circ S_1) \mid S')$ measures how strongly the decision function $\psi$ (trained on the augmented set $S' \circ S_1$) correlates with the specific statistical fluctuations $\phi$ of the set $S_1$ that was added to it.
In stable algorithms, adding $S_1$ to the training set should have negligible impact on the decision, leading to near-zero covariance. In unstable algorithms (like Majority near the decision boundary), the correlation is high, and the lemma precisely quantifies this instability.

Indeed, Majority is multiplicatively separable in the sense discussed above.
\begin{corollary}\label{cor:majFactor}
The majority algorithm with uniformly random labels admits the fold-covariance
	\begin{align*}
		\mathrm{Cov}(\widehat{L}_1, \widehat{L}_2) 
		&= 4\cdot  \mathbb{E}_{Y} \left[ \mathrm{Cov}\bigg(\frac{X_m}{m}, \mathbb{1}_{\{X_m > \theta - Y\}} \bigg|\, Y\bigg)^2 \right]
	\end{align*}
    where $\theta:=(n-m)/2$, $X_t\sim \mathrm{Bin}(t,1/2)$ and $Y\sim \mathrm{Bin}(n-2m,1/2)$.
\end{corollary}
\begin{proof}
	Let $X_i:=\sum_{(x,y)\in S_i} y$ and $Y_{-i}:=\sum_{(x,y)\in S \backslash S_i} y$. It holds that
    \begin{align*}
    \widehat{L}_i &= \frac{X_i}{m} \mathbb{1}_{\{Y_{-i} \le \theta\}}+ \frac{m-X_i}{m} \mathbb{1}_{\{Y_{-i} > \theta\}} 
    = \frac{X_i}{m} \left(1 - \mathbb{1}_{\{Y_{-i} > \theta\}}\right) + \bigg[1-\frac{X_i}{m} \bigg] \mathbb{1}_{\{Y_{-i} > \theta\}} \\
    &= \bigg[1-\frac{2X_i}{m} \bigg](\mathbb{1}_{\{Y_{-i} > \theta\}}-1/2)+\frac12.
\end{align*}
	This matches the form of Lemma~\ref{lem:factor} with:
	\begin{itemize}
		\item Test Statistic: $\phi(S_i) = 1-\frac{2X_i}{m}$.
		\item Decision function: $\psi(S_{-i}) = \mathbb{1}_{\{Y_{-i} > \theta\}} - \frac{1}{2}$.
		\item For uniformly random labels, $\mathbb{E}[X_i] = m/2$, so $\mu = \mathbb E[\widehat L_i]=\frac{1}{2}$ and $\mathbb{E}[\phi(S_i)] = 0$.
	\end{itemize}
	Since shifting $\phi,\psi$ by constants does not change their covariance, the statement follows.
\end{proof}

\subsubsection{Precise Analysis of the MSE}
For the remainder of this section, let us denote (with slight abuse of notation) $\Cov(n,m)\equiv \Cov(\widehat L_1,\widehat L_2)$ to highlight the roles of $n$ and $m$.

With the structural decomposition of the covariance established, we now turn to a precise quantitative analysis, transforming the probabilistic factorization into an exact combinatorial expression and, finally, deriving sharp non-asymptotic bounds on the MSE.

A key analytical step is the reduction of this covariance term to a point evaluation of the binomial probability mass function. 
\begin{lemma}
In the setting of Corollary~\ref{cor:majFactor}, it holds that
\begin{equation*}
    \mathrm{Cov}\left(\frac{X_m}{m}, \mathbb{1}_{\{X_m > \theta - Y\}} \;\bigg|\; Y\right) = \frac{1}{4} \mathbb{P}\big(X_{m-1} = \lfloor \theta - Y \rfloor \mid Y\big).
\end{equation*}
\end{lemma}
\begin{proof}
See Appendix~\ref{app:maj_exact-combinatorial}.
\end{proof}
This relation is significant because it transforms the stability analysis from a moment-estimation problem into a counting problem. 

Substituting this probability mass directly into the factorization formula converts the expectation over the shared training set count $Y$ into a discrete convolution of binomial coefficients.
\begin{theorem}[Exact Combinatorial Form of Fold-Covariance]\label{thm:exactComb}
For $1\le m\le n/2$, $m|n$, we have
	\[
	\Cov(n,m)=
	2^{-n}\sum_{j=0}^{m-1}\binom{m-1}{j}^2
	\binom{n-2m}{\lfloor (n-m)/2\rfloor-j}.
	\]
\end{theorem}
\begin{proof}
See Appendix~\ref{app:maj_exact-combinatorial}.
\end{proof}
While Theorem~\ref{thm:exactComb} provides an exact description of the covariance, extracting explicit scaling laws requires a delicate non-asymptotic analysis.
Indeed, up to normalization, the summand in Theorem~\ref{thm:exactComb} is a product of binomial probability masses, so the asymptotics depend on pointwise approximations to these masses across the relevant range of \(j\).
Consequently, our analysis relies fundamentally on a local limit theorem (LLT) to replace discrete binomial coefficients with Gaussian density functions, managing the approximation error in two distinct regimes.

The proof strategy bifurcates based on the relative size of the fold $m$:
\begin{enumerate}
    \item \textbf{The Sublinear Regime ($m \ll n$):} In this setting, the summation range over the validation fold statistics is narrow relative to the scale of the training set fluctuations. Due to this scale separation, the probability mass function governing the training decision boundary is locally flat over the effective support of the validation statistic. This allows us to approximate the training term by a single \textit{central coefficient}, decoupling the interaction.
    \item \textbf{The Large $m$ Regime:} When $m$ is comparable to $n$, the scale separation vanishes and we cannot fix the training term. Instead, we approximate the entire summation as a \textit{triple convolution} of Gaussian probability densities. Since the support is broad, the discrete sum converges rapidly to a Gaussian integral. The analysis here focuses on controlling the accumulated error from the Gaussian approximation across a large number of terms.
\end{enumerate}

The above approach establishes the following result, a more explicit version of which (including precise constants) is provided in Appendix~\ref{App:mainPrecise}.

\begin{theorem}[Majority Asymptotics]\label{thm:majority}
	Throughout, let $n\ge 2$ and $m|n$. For Majority in the uniformly random-label classification model under 0--1 loss, the following statements hold:
	
	\medskip
	\noindent\textbf{(A)} Uniformly over all \(m=m_n\) with \(m=\omega_n(1)\), and \(m\le n/2\),
	\[
	\Cov(n,m)
	= M_{n,m}
		 +o_n(M_{n,m})
    =\Theta\left(\frac{\sqrt{k}}{n}\right) \quad\text{where}\quad M_{n,m}=\frac{1}{2\pi\sqrt{(m-1)(2n-3m)}}.
	\]
    
    	\noindent{\bf (B) Monotonicity and minimizer of the Covariance.}

Let
\(
    \mathcal M_n:=\{m\in\mathbb N:m\mid n,\ 1\le m\le n/2\}.
\)
For all sufficiently large \(n\), if \(m_1,m_2\in\mathcal M_n\) satisfy
\(m_1<m_2\le n/3\), then
\[
    \operatorname{Cov}(n,m_1)>\operatorname{Cov}(n,m_2).
\]
Moreover, along subsequences with \(6\mid n\),
\[
    \operatorname{Cov}(n,n/3)<\operatorname{Cov}(n,n/2).
\]
Consequently, along subsequences with \(3\mid n\), the covariance-minimizing
choice is \(k=3\).

\medskip
\noindent{\bf (C) Minimizer of the MSE.}
Along subsequences of even \(n\), the MSE is minimized by \(k=2\) for all
sufficiently large \(n\). Moreover,
\[
    \min_{k\mid n}\operatorname{MSE}_{\mathrm{CV}}^{(k)}
    =
    \operatorname{MSE}_{\mathrm{CV}}^{(2)}
    =
    \frac{1}{4n}+\frac{1}{2\pi n}+o(n^{-1}).
\]
\end{theorem}
\begin{proof}\emph{Proof of part (C).}
By Proposition~\ref{prop:majMSE},
\(
    \operatorname{MSE}_{\mathrm{CV}}^{(k)}
    =
    \frac{k-1}{k}\Cov(n,n/k)+\frac{1}{4n}.
\)
The second term is independent of \(k\). On the range \(m\le n/3\), both
\(\Cov(n,m)\) and the prefactor \((k-1)/k=1-m/n\) decrease with \(m\), so their product is asymptotically bounded below by its continuous minimum at \(m=n/3\). Hence, by part~(B), it remains only to compare \(m=n/3\) and
\(m=n/2\).
Using part~(A) and the endpoint asymptotic at \(m=n/2\),
\[
    \frac23\Cov(n,n/3)
    =
    \frac{1}{\sqrt 3\pi n}+o(n^{-1}),
    \qquad
    \frac12\Cov(n,n/2)
    =
    \frac{1}{2\pi n}+o(n^{-1}).
\]
Since \(1/(2\pi)<1/(\sqrt 3\pi)\), the full MSE is minimized by \(k=2\) for all sufficiently
large \(n\).
\end{proof}

We observe that the MSE scales as $\sqrt{k}/n$. In this setting, it is therefore advantageous to choose as few folds as possible. Notably, \emph{hypothesis stability--based bounds} are not sufficiently fine-grained here: they incorrectly predict the MSE to increase when $k$ decreases, see Section~\ref{sec:lim}.

\subsection{A Minimax Lower Bound for Cross-Validation with ERM Algorithms}\label{sec:minimax_maj}

The answer to Question~\ref{q2} follows as a corollary of the preceding analysis of the
Majority algorithm. Let $\mathcal A$ be an ERM algorithm over a nontrivial hypothesis
class $\mathcal H$, meaning that there exist $h_0,h_1\in\mathcal H$ and
$x_0\in\mathcal X$ such that $h_0(x_0)\neq h_1(x_0)$. Consider the distribution
$\mathcal D$ supported on $x_0$, with the label drawn uniformly from
$\{h_0(x_0),h_1(x_0)\}$.
For this distribution and the 0--1 loss, empirical risk minimization reduces to choosing
a classifier whose prediction at $x_0$ agrees with the majority label in the sample.
Consequently, away from exact ties, the induced prediction at \(x_0\) follows the
binary Majority rule. Tie-breaking may depend on the ERM, but tie events affect
only lower-order terms.
Thus, the lower bound obtained above for Majority applies to any ERM over a nontrivial
hypothesis class.

\begin{corollary}\label{cor:minimax}
For any ERM algorithm $\mathcal{A}$ for classification with 0--1 loss such that $|\mathcal H| \ge 2$, it holds that 
\[
\mathfrak{R}_{\mathrm{CV}}(\mathcal{A}) = \Omega\left(\frac{\sqrt{k^*}}{n}\right),
\]

where $k^*$ is the number of folds that achieves the minimax optimum.
\end{corollary}

This result shows that, in the distribution-free setting, no ERM algorithm with \(k^*(n)=\omega_n(1)\) can 
utilize all $n$ samples as efficiently as an independent validation set of the same size, whose mean-squared error decreases at the optimal rate of order $1/n$. Importantly, this result is agnostic to the specific value of \(k^*\). While for Majority under random labels the optimal number of folds is \(k=2\), for other algorithms the minimax optimum may be attained at \(k^*\) growing with \(n\), in which case the lower bound becomes even stronger.

Finally, as a straightforward corollary, our precise analysis of Majority allows us to make the gap between CV and using an independent hold-out set explicit.
\begin{corollary}[Constant-factor gap for ERM]\label{cor:constant-gap}
Assume \(n\) is even. For any ERM algorithm $\mathcal{A}$ for classification with 0--1 loss such that $|\mathcal H| \ge 2$, there exists a data distribution \(\D\) such that cross-validation incurs a multiplicative gap of
\[
    \frac{
    \min_{k\mid n}\MSE_{\mathrm{CV}}^{(k)}(\mathcal A,\mathcal D)
    }{
    \MSE_{\mathrm{val}}^{(n)}(\mathcal A,\mathcal D)
    }=1+\frac{2}{\pi}+O(n^{-1})
\]
where \(\MSE_{\mathrm{val}}^{(n)}\) denotes the MSE of an independent validation set of size \(n\), and numerically, \(1+\frac{2}{\pi}\approx 1.637\).
\end{corollary}

\begin{proof}
It suffices to consider the one-point random-label distribution used in the proof of
Corollary~\ref{cor:minimax}. Under this distribution, any nontrivial ERM reduces to
Majority, for which Theorem~\ref{thm:majority}(C) shows that the MSE is minimized at \(k=2\), and
\(
    \operatorname{MSE}_{\mathrm{CV}}^{(2)}
    =
    \frac{1}{4n}+\frac{1}{2\pi n}+O(n^{-2}).
\)
An independent validation set of size \(n\) has MSE \(\mathbb{E}_{Y \sim \mathrm{Bin}(n, 1/2)}[(1/2 - Y/n)^2]=1/(4n)\). Taking the ratio yields the result.
\end{proof}

\subsection{Limitations of Existing Cross-Validation Theory: the Majority Benchmark}\label{sec:lim}

\paragraph{Arbitrarily loose sufficient conditions.}

We begin by examining the notion of hypothesis stability (Definition~\ref{def:hyp_stab}). Observe that the majority algorithm changes its output only when the sample is close to a tie. In the extreme case $k=n$ (leave-one-out CV), Majority flips its prediction only when the full sample has a tie or a one-vote margin, depending on
parity and on the removed label. Since the probability of a tie is $\Theta(1/\sqrt{n})$, Majority is $(\Theta(1/\sqrt{n}),1)$-hypothesis-stable.

For a general number of folds $k$, each fold has size $m=n/k$. By concentration of measure, the number of ones in a typical fold is $m/2 \pm O(\sqrt{m})$. For the output of Majority to change when a typical fold is removed, the full sample must therefore have a label count within $O(\sqrt{m})$ of $n/2$. The probability of this event is $\Theta(\sqrt{m}/\sqrt{n})=\Theta(1/\sqrt{k})$. Consequently, Majority is $(\Theta(1/\sqrt{k}),m)$-hypothesis-stable. In particular, the stability of Majority deteriorates as the number of folds decreases.

If one extends the stability-based bounds of \cite{rogers,devroye,kearns,bousquet} from leave-one-out to general $k$-fold cross-validation, the resulting guarantees retain the same qualitative form,
\[
L \le \widehat L_{\mathrm{CV}}^{(k)} + O(\sqrt{\beta}),
\]
where $\beta$ denotes the hypothesis stability parameter. However, our analysis shows that CV provides more accurate risk estimates for Majority precisely when the algorithm is \emph{less} stable. This mismatch is most pronounced at $k=2$, where the stability-based bound differs from the true mean-squared error by a factor of $\Theta(n)$.

Similarly, \citet{kumar2013near} introduce a refined notion of stability, termed \emph{loss stability} (Definition~2 in their paper, distinct from Definition~\ref{def:lossstab} here), parameterized by $\gamma$. For the Majority algorithm, this notion effectively reduces to hypothesis stability with $m=1$, except that the deviation inside the expectation is squared rather than taken in absolute value. By the discussion above, this yields $\gamma=\Theta(1/\sqrt{n})$ for Majority.

Instantiating Theorem~1 of \citet{kumar2013near} with the Majority algorithm, we find that the resulting upper bound on the MSE of the $k$-fold CV estimator is dominated by the term $(1-\tfrac{1}{k})\gamma$. Consequently, the bound becomes looser as $k$ decreases. In the extreme case $k=2$, the predicted upper bound exceeds the true MSE of CV for Majority by a factor of $\Theta(\sqrt{n})$ ($\mathrm{MSE}_{\mathrm{CV}}^{(2)}=\Theta(1/n)$ by Theorem~\ref{thm:majority}, and the bound predicts $\Theta(1/\sqrt{n})$).

\paragraph{Relative rather than absolute guarantees.}

\cite{kearns} compare leave-one-out cross-validation to empirical error and show that, for stable algorithms, the CV estimate is never much worse than the empirical error. However, this type of relative guarantee can be misleading when the quantity of interest is the MSE of risk estimation.

In particular, for Majority under leave-one-out CV we showed that $\mathrm{MSE}_{\mathrm{CV}}^{(n)}=\Theta(1/\sqrt{n})$. 
On the other hand, the empirical error is
\(
    \widehat L_{\rm emp}
    =
    \min(Y/n,1-Y/n),
\)
and therefore
\(
    \left(\widehat L_{\rm emp}-\frac12\right)^2
    =
    \left(Y/n-\frac12\right)^2
\)
which leads to an MSE of
\(
    \frac{1}{4n}.
\)

Thus, even though leave-one-out CV is guaranteed by \citet{kearns} to perform comparably to empirical error in a relative sense, its absolute performance in terms of MSE is worse by a factor of $\Theta(\sqrt{n})$. This gap is substantial, and it arises for an algorithm that is arguably among the most stable nontrivial learning rules.


\subsection{Mathematical Inaccuracies in Prior Analyses}\label{sec:errors}
We also identify two concrete technical issues in prior analyses of cross-validation. These issues are qualitative rather than quantitative: the affected results make claims about fundamental properties of CV estimation that fail as stated.

First, Theorem 5.3 of \cite{kearns} claims a necessity of loss stability for accurate leave-one-out estimation. The proof overlooks possible cancellation between positive and negative deviations in the CV error. Lemma~\ref{lem:anticorr} gives a counterexample; see Appendix~\ref{app:err}.

Second, the main theorem of \cite{kale2011cross} uses a conditional-expectation identity that is false in general: the empirical term uses a fixed validation fold, while the population term averages over an independent test point. Consequently, the advertised MSE bound does not follow from the proof; see Appendix~\ref{app:errKale}.

\section{Conclusion}

Our results show that the dependence induced by data reuse in cross-validation can impose
a genuine statistical cost, even for very simple ERM procedures. As a natural extension of Corollary~\ref{cor:minimax}, it would be compelling to investigate which combined properties of algorithms and data distributions can yield improved minimax rates (or even attain the optimal $1/n$ rate) in settings beyond the distribution-free case.


\newpage
\bibliography{colt2026final}

@article{specific2,
  title={Choice of V for V-fold cross-validation in least-squares density estimation},
  author={Arlot, Sylvain and Lerasle, Matthieu},
  journal={Journal of Machine Learning Research},
  volume={17},
  number={208},
  pages={1--50},
  year={2016}
}

@article{specific1,
  author = {Alain Celisse},
 journal = {The Annals of Statistics},
 number = {5},
 pages = {1879--1910},
 publisher = {Institute of Mathematical Statistics},
 title = {OPTIMAL CROSS-VALIDATION IN DENSITY ESTIMATION WITH THE L²-LOSS},
 urldate = {2026-02-05},
 volume = {42},
 year = {2014}
}

@inproceedings{kearns,
  title={Algorithmic stability and sanity-check bounds for leave-one-out cross-validation},
  author={Kearns, Michael and Ron, Dana},
  booktitle={Proceedings of the tenth annual conference on Computational learning theory},
  pages={152--162},
  year={1997}
}

@article{rogers,
  title={A finite sample distribution-free performance bound for local discrimination rules},
  author={Rogers, William H and Wagner, Terry J},
  journal={The Annals of Statistics},
  pages={506--514},
  year={1978},
  publisher={JSTOR}
}

@article{devroye,
  title={Distribution-free performance bounds for potential function rules},
  author={Devroye, Luc and Wagner, Terry},
  journal={IEEE Transactions on Information Theory},
  volume={25},
  number={5},
  pages={601--604},
  year={1979},
  publisher={IEEE}
}

@article{devroye2,
  title={Distribution-free inequalities for the deleted and holdout error estimates},
  author={Devroye, Luc and Wagner, Terry},
  journal={IEEE Transactions on Information Theory},
  volume={25},
  number={2},
  pages={202--207},
  year={1979},
  publisher={IEEE}
}

@article{bousquet,
  title={Stability and generalization},
  author={Bousquet, Olivier and Elisseeff, Andr{\'e}},
  journal={Journal of machine learning research},
  volume={2},
  number={Mar},
  pages={499--526},
  year={2002}
}

@article{bengioUnbiased,
  title={No unbiased estimator of the variance of k-fold cross-validation},
  author={Bengio, Yoshua and Grandvalet, Yves},
  journal={Journal of Machine Learning Research},
  volume={5},
  number={Sep},
  pages={1089--1105},
  year={2004}
}

@inproceedings{blum,
  title={Beating the hold-out: Bounds for k-fold and progressive cross-validation},
  author={Blum, Avrim and Kalai, Adam and Langford, John},
  booktitle={Proceedings of the Twelfth Annual Conference on Computational Learning Theory},
  pages={203--208},
  year={1999}
}

@article{blumer,
  title={Learnability and the Vapnik-Chervonenkis dimension},
  author={Blumer, Anselm and Ehrenfeucht, Andrzej and Haussler, David and Warmuth, Manfred K},
  journal={Journal of the ACM (JACM)},
  volume={36},
  number={4},
  pages={929--965},
  year={1989},
  publisher={ACM New York, NY, USA}
}

@article{vapnik,
  title={On the Uniform Convergence of Relative Frequencies of Events to Their Probabilities},
  author={Vapnik, Vladimir N and Chervonenkis, Alexey Y},
  journal={Theory of Probability \& Its Applications},
  volume={16},
  number={2},
  pages={264--280},
  year={1971},
  publisher={SIAM}
}

@inproceedings{nachum2024fantastic,
  title={Fantastic generalization measures are nowhere to be found},
  author={Gastpar, Michael and Nachum, Ido and Shafer, Jonathan and Weinberger, Thomas},
  booktitle={The Twelfth International Conference on Learning Representations},
  year={2024}
}

@inproceedings{jiangfantastic,
  title={Fantastic Generalization Measures and Where to Find Them},
  author={Jiang, Yiding and Neyshabur, Behnam and Mobahi, Hossein and Krishnan, Dilip and Bengio, Samy},
  booktitle={International Conference on Learning Representations},
  year={2019}
}

@article{gastpar2024algorithms,
  title={Which Algorithms Have Tight Generalization Bounds?},
  author={Gastpar, Michael and Nachum, Ido and Shafer, Jonathan and Weinberger, Thomas},
  journal={Advances in Neural Information Processing Systems},
  volume={38},
  pages={36554--36589},
  year={2026}
}

@inproceedings{anthony,
  title={Cross-validation for binary classification by real-valued functions: theoretical analysis},
  author={Anthony, Martin and Holden, Sean B},
  booktitle={Proceedings of the eleventh annual conference on Computational learning theory},
  pages={218--229},
  year={1998}
}

@article{arlot2010survey,
  title={A survey of cross-validation procedures for model selection},
  author={Arlot, Sylvain and Celisse, Alain},
    journal={Statistics Surveys},
   volume={4},
  year={2010}
}

@article{montecarlo,
  title={Cross-validation of regression models},
  author={Picard, Richard R and Cook, R Dennis},
  journal={Journal of the American Statistical Association},
  volume={79},
  number={387},
  pages={575--583},
  year={1984},
  publisher={Taylor \& Francis}
}

@article{combinatCV,
  title={Linear model selection by cross-validation},
  author={Shao, Jun},
  journal={Journal of the American statistical Association},
  volume={88},
  number={422},
  pages={486--494},
  year={1993},
  publisher={Taylor \& Francis}
}

@article{bengio2,
  title={Inference for the generalization error},
  author={Nadeau, Claude and Bengio, Yoshua},
  journal={Advances in neural information processing systems},
  volume={12},
  year={1999}
}

@inproceedings{kumar2013near,
  title={Near-optimal bounds for cross-validation via loss stability},
  author={Kumar, Ravi and Lokshtanov, Daniel and Vassilvitskii, Sergei and Vattani, Andrea},
  booktitle={International Conference on Machine Learning},
  pages={27--35},
  year={2013},
  organization={PMLR}
}

@inproceedings{kale2011cross,
  title={Cross-validation and mean-square stability},
  author={Kale, Satyen and Kumar, Ravi and Vassilvitskii, Sergei},
  booktitle={Innovations in Computer Science (ICS)},
  pages={487--495},
  year={2011}
}

@article{dziugaite2020search,
  title={In search of robust measures of generalization},
  author={Dziugaite, Gintare Karolina and Drouin, Alexandre and Neal, Brady and Rajkumar, Nitarshan and Caballero, Ethan and Wang, Linbo and Mitliagkas, Ioannis and Roy, Daniel M},
  journal={Advances in Neural Information Processing Systems},
  volume={33},
  pages={11723--11733},
  year={2020}
}

@book{petrov2012sums,
  title={Sums of independent random variables},
  author={Petrov, Valentin V},
  volume={82},
  year={2012},
  publisher={Springer Science \& Business Media}
}

@book{UML,
author = {Shalev-Shwartz, Shai and Ben-David, Shai},
title = {Understanding Machine Learning: From Theory to Algorithms},
year = {2014},
isbn = {1107057132},
publisher = {Cambridge University Press},
address = {USA}
}

@article{bayle,
  title={Cross-validation confidence intervals for test error},
  author={Bayle, Pierre and Bayle, Alexandre and Janson, Lucas and Mackey, Lester},
  journal={Advances in Neural Information Processing Systems},
  volume={33},
  pages={16339--16350},
  year={2020}
}

\newpage

\appendix

\section{Majority Algorithm}\label{app:majority}

Throughout this section, we consider the following setup.
\subsection{Setup and Notation}
	For $1\le m\le n/2$, $m|n$, let $N:=n-2m$ and define
	\[
	E(n, m):=
	2^{-(n-2)}\sum_{j=0}^{m-1}\binom{m-1}{j}^2
	\binom{n-2m}{\lfloor (n-m)/2\rfloor-j}
	\]
    such that $\Cov(\widehat L_1, \widehat L_2)=E(n,m)/4$ (see Thm.~\ref{thm:maj_exact-combinatorial} for details).
    We also denote $\Cov(n,m) \equiv\Cov(\widehat L_1, \widehat L_2)$ to highlight the roles of $n,m$.
    
	Let $\mathsf B_r\sim\mathrm{Bin}(r,\tfrac12)$ with pmf
	$p_r(t)=2^{-r}\binom{r}{t}$, and denote the Gaussian proxy
	\[
	g_r(t):=\sqrt{\frac{2}{\pi r}}\,
	\exp\!\Big(-\frac{(2t-r)^2}{2r}\Big)
	\]
	and central binomial mass
	\[
	S_{r}:=2^{-2r}\binom{2r}{r}
		\]

\subsection{Main Theorem}\label{App:mainPrecise}
\begin{theorem}[Fold-Covariance of the Majority Algorithm]
	Throughout, let $n\ge 2$ and $m|n$.
	
	\medskip

	\noindent\textbf{(A) Binomial form.} 
    One has
\[
\Cov(n,m)
=
S_{m-1}\frac{1}{2\sqrt{\pi(2n-3m)}}
+
O\left(n^{-3/2}+\frac{m^{3/2}}{n^{5/2}}\right),
\]
uniformly for all \(1\le m\le n/3\), where
\(
    S_{m-1}:=2^{-(2m-2)}\binom{2m-2}{m-1}.
\)

	\medskip
    	
	\noindent{\bf (B) Exact expression for $m=1$.}
It holds that
	\[
	\Cov(n,1)=2^{-n}\binom{\,n-2\,}{\big\lfloor \frac{n-1}{2}\big\rfloor}=\sqrt{\frac{1}{8\pi(n-2)}}+O\Big(\frac{1}{n^{3/2}}\Big)
	=\sqrt{\frac{1}{8\pi n}}+O\Big(\frac{1}{n}\Big).
	\]
    \medskip
    
	\noindent{\bf (C) Endpoint asymptotic for $m=n/2$.}
	It holds that
		\[
	\Cov\Big(n,\tfrac{n}{2}\Big)=\frac{1}{\pi(n-2)}+O\Big(\frac{1}{n^{2}}\Big)
	=\frac{1}{\pi n}+O\Big(\frac{1}{n^{2}}\Big).
	\]
        \medskip
    
	\noindent\textbf{(D) Large $m$ regime.}
    Uniformly over all integer sequences \(m=m_n\) satisfying
\(m=\omega_n(1)\) and \(m\le n/3\),
		\[
		\Cov(n,m)=\frac{1}{2\pi\sqrt{(m-1)(2n-3m)}}+O\bigg(\frac{1}{\sqrt{n}m^{3/2}}\bigg).
		\]
        
        	\medskip

    	\noindent{\bf (E) Monotonicity and minimizer.}
	Let
\(
    \mathcal M_n:=\{m\in\mathbb N:m\mid n,\ 1\le m\le n/2\}.
\)
For all sufficiently large \(n\), if \(m_1,m_2\in\mathcal M_n\) satisfy
\(m_1<m_2\le n/3\), then
\[
    \operatorname{Cov}(n,m_1)>\operatorname{Cov}(n,m_2).
\]
Moreover, along subsequences with \(6\mid n\),
\[
    \operatorname{Cov}(n,n/3)<\operatorname{Cov}(n,n/2).
\]
Consequently, along subsequences with \(3\mid n\), the covariance-minimizing
choice is \(k=3\).
\end{theorem}
\begin{proof}
This is a consequence of collecting the results of \cref{thm:maj_small-k-mono,thm:maj_sublin,thm:maj_large-m,thm:maj_mono-minimizer}.
\end{proof}

\bigskip\bigskip

\subsection{Technical Lemmas}
Let us first state a few technical results.

\begin{lemma}[Triple Gaussian Product]\label{lem:twoG}
	Let $P(j) := g_{m-1}(j)^2\,g_{N}(\ell-j)$. With the parameters
	\[
	\alpha:=\frac{4}{m-1},\qquad
	\beta :=\frac{2}{N},\qquad
	\mu:=\frac{\alpha\cdot\frac{m-1}{2}+\beta\cdot\frac{m}{2}}{\alpha+\beta}
	=\frac{(m-1)(2N+m)}{2(2N+m-1)},
	\]
	the product $P(j)$ can be written as:
	\begin{align*}
		P(j)
		&= \left( \frac{2}{\pi(m-1)} \sqrt{\frac{2}{\pi N}} \right)
		\exp\!\left( - \frac{1}{2N+m-1} \right)
		\exp\!\Big(-(\alpha+\beta)(j-\mu)^2\Big).
	\end{align*}
	Furthermore, the sum of the rates is
	\[
	\alpha+\beta = \frac{2(2N+m-1)}{(m-1)N}.
	\]
\end{lemma}

\begin{proof}
	Recall
	\[
	g_r(t):=\sqrt{\frac{2}{\pi r}}\,
	\exp\!\Big(-\frac{(2t-r)^2}{2r}\Big).
	\]
	Let $N:=n-2m$ and $\ell=(n-m)/2$.
	
	We first write out the terms. Let $a:=(m-1)/2$ and $\alpha:=4/(m-1)$.
	\[
	g_{m-1}(j)^2 = \left(\sqrt{\frac{2}{\pi(m-1)}}\right)^2
	\exp\!\Big(-2 \cdot \frac{2}{m-1}(j-\tfrac{m-1}{2})^2\Big)
	= \frac{2}{\pi(m-1)}\,e^{-\alpha(j-a)^2}.
	\]
	For the second term, let $b:=m/2$ and $\beta:=2/N$. The exponent's center is
	$\ell-j - \frac{N}{2} = \frac{n-m}{2} - j - \frac{n-2m}{2} = \frac{m}{2} - j = -(j-b)$.
	Thus,
	\[
	g_N(\ell-j) = \sqrt{\frac{2}{\pi N}}\,
	\exp\!\Big(-\frac{2}{N}(\ell-j-\tfrac N2)^2\Big)
	= \sqrt{\frac{2}{\pi N}}\,e^{-\beta(j-b)^2}.
	\]
	The product is
	\[
	g_{m-1}(j)^2\,g_N(\ell-j) =
	\underbrace{\left( \frac{2}{\pi(m-1)} \sqrt{\frac{2}{\pi N}} \right)}_{:=C_{\text{prod}}}\,
	\exp\{-\alpha(j-a)^2-\beta(j-b)^2\}.
	\]
	We complete the square for the exponential terms:
	\[
	-\alpha(j-a)^2-\beta(j-b)^2
	=-(\alpha+\beta)(j-\mu)^2-\frac{\alpha\beta}{\alpha+\beta}(a-b)^2,
	\]
	where $\mu:=(\alpha a+\beta b)/(\alpha+\beta)$ is as stated in the lemma.
	The constant term in the exponent depends on $a-b = (m-1)/2 - m/2 = -1/2$.
	\[
	\frac{\alpha\beta}{\alpha+\beta}(a-b)^2
	= \frac{1}{4} \cdot \frac{\frac{4}{m-1} \cdot \frac{2}{N}}{\frac{4}{m-1} + \frac{2}{N}}
	= \frac{1}{4} \cdot \frac{8/((m-1)N)}{(4N+2m-2)/((m-1)N)}
	 = \frac{1}{2N+m-1}.
	\]
	We also compute
	\[
	\alpha+\beta = \frac{4}{m-1} + \frac{2}{N}  = \frac{2(2N+m-1)}{(m-1)N}.
	\]
	Combining these results yields the displayed formula.
\end{proof}

\begin{lemma}[Poisson summation for Gaussians]\label{lem:PSF}
	Let $\gamma>0$ and $\mu\in\mathbb R$. Define
	\[
	f_{\gamma,\mu}(x):=e^{-\gamma(x-\mu)^2}.
	\]
	Then
	\begin{equation}\label{eq:PSF}
		\sum_{j\in\mathbb Z} f_{\gamma,\mu}(j)
		=\sqrt{\frac{\pi}{\gamma}}\,
		\sum_{t\in\mathbb Z} e^{-\pi^2 t^2/\gamma}\,e^{-2\pi i t\mu}.
	\end{equation}
\end{lemma}

\begin{proof}
	Let $\mathcal P_{\gamma,\mu}(x):=\sum_{j\in\mathbb Z} f_{\gamma,\mu}(x+j)$ be the periodization (absolutely and uniformly convergent on $\mathbb R$). Then $\mathcal P_{\gamma,\mu}$ is $1$–periodic and belongs to $C^\infty$. Its complex Fourier series is
	\[
	\mathcal P_{\gamma,\mu}(x)=\sum_{t\in\mathbb Z} c_t\,e^{2\pi i t x},
	\qquad
	c_t=\int_0^1\mathcal P_{\gamma,\mu}(x)\,e^{-2\pi i t x}\,dx.
	\]
	By absolute convergence we may integrate termwise:
	\[
	c_t=\sum_{j\in\mathbb Z}\int_0^1 e^{-\gamma(x+j-\mu)^2}\,e^{-2\pi i t x}\,dx
	=\int_{\mathbb R} e^{-\gamma(y-\mu)^2}\,e^{-2\pi i t y}\,dy
	=: \widehat f_{\gamma,\mu}(t),
	\]
	after the change of variables $y=x+j$. The Gaussian Fourier transform is standard:
	\[
	\widehat f_{\gamma,\mu}(t)
	=e^{-2\pi i t\mu}\int_{\mathbb R}e^{-\gamma z^2}e^{-2\pi i t z}\,dz
	=e^{-2\pi i t\mu}\sqrt{\frac{\pi}{\gamma}}\;e^{-\pi^2 t^2/\gamma}.
	\]
	Thus
	\[
	\mathcal P_{\gamma,\mu}(x)
	=\sqrt{\frac{\pi}{\gamma}}\sum_{t\in\mathbb Z}e^{-\pi^2 t^2/\gamma}\,e^{2\pi i t(x-\mu)}.
	\]
	Evaluating at $x=0$ gives
	\[
	\sum_{j\in\mathbb Z} f_{\gamma,\mu}(j)
	=\mathcal P_{\gamma,\mu}(0)
	=\sqrt{\frac{\pi}{\gamma}}\sum_{t\in\mathbb Z}e^{-\pi^2 t^2/\gamma}\,e^{-2\pi i t\mu},
	\]
	which is \eqref{eq:PSF}.
\end{proof}

\begin{lemma}[Lattice sum of the triple Gaussian]\label{lem:triple-sum-theta}
	With $N=n-2m$, $\ell=(n-m)/2$, and the parameters
	\[
	\alpha=\frac{4}{m-1},\quad \beta=\frac{2}{N},\quad
	\mu=\frac{(m-1)(2N+m)}{2(2N+m-1)},
	\]
	we have the exact identity
	\begin{align}
		\sum_{j\in\mathbb Z} g_{m-1}(j)^2\,g_N(\ell-j)
		&= \frac{2}{\pi}\cdot \frac{1}{\sqrt{(m-1)(2N+m-1)}}\,
		e^{-\frac{1}{\,2N+m-1\,}}\;
		\Theta_{n,m},\label{eq:triple-theta}\\
		\Theta_{n,m}
		&:= \sum_{t\in\mathbb Z}
		\exp\!\Big(-\pi^2 t^2/(\alpha+\beta)\Big)\,
		\exp\!\big(-2\pi i t \mu\big).\label{eq:theta-def}
	\end{align}
	Equivalently, using $\alpha+\beta=\dfrac{2(2N+m-1)}{(m-1)N}$,
	\begin{align}
		\sum_{j\in\mathbb Z} &g_{m-1}(j)^2\,g_N(\ell-j)\notag\\
		&= \frac{2}{\pi}\cdot \frac{1}{\sqrt{(m-1)(2n-3m-1)}}\,
		e^{-\frac{1}{\,2n-3m-1\,}}\;
		\sum_{t\in\mathbb Z}
		\exp\!\Big(-\frac{\pi^2 t^2 (m-1)N}{2(2n-3m-1)}\Big)\,
		e^{-2\pi i t \mu}\label{eq:triple-theta-simplified}.
	\end{align}
\end{lemma}

\begin{proof}
	By Lemma~\ref{lem:twoG}, we have
	\[
	g_{m-1}(j)^2\,g_N(\ell-j)
	= C_{\text{prod}} \cdot
	e^{-\frac{1}{\,2N+m-1\,}}\;
	\exp\!\Big(-(\alpha+\beta)(j-\mu)^2\Big),
	\]
	where $C_{\text{prod}} = \frac{2}{\pi(m-1)} \sqrt{\frac{2}{\pi N}}$.
	Summing over $j\in\mathbb Z$ and applying Lemma~\ref{lem:PSF} with
	$\gamma:=\alpha+\beta$, we obtain
	\begin{align*}
		\sum_{j\in\mathbb Z} g_{m-1}(j)^2\,g_N(\ell-j)
		&= C_{\text{prod}} \cdot e^{-\frac{1}{\,2N+m-1\,}}
		\sum_{j\in\mathbb Z} e^{-(\alpha+\beta)(j-\mu)^2} \\
		&= C_{\text{prod}} \cdot e^{-\frac{1}{\,2N+m-1\,}}
		\cdot \sqrt{\frac{\pi}{\alpha+\beta}}\sum_{t\in\mathbb Z}
		e^{-\pi^2 t^2/(\alpha+\beta)}e^{-2\pi i t\mu}.
	\end{align*}
	We now compute the combined prefactor. Using $\alpha+\beta = \frac{2(2N+m-1)}{(m-1)N}$ from Lemma~\ref{lem:twoG}:
	\begin{align*}
		C_{\text{prod}}\;\sqrt{\frac{\pi}{\alpha+\beta}}
		&= \left( \frac{2}{\pi(m-1)} \sqrt{\frac{2}{\pi N}} \right)
		\cdot \sqrt{\frac{\pi (m-1) N}{2(2N+m-1)}} \\
		&= \left( \frac{2\sqrt{2}}{\pi^{3/2} (m-1) \sqrt{N}} \right)
		\cdot \left( \frac{\sqrt{\pi} \sqrt{m-1} \sqrt{N}}{\sqrt{2} \sqrt{2N+m-1}} \right) \\
		&= \frac{2}{\pi \sqrt{m-1} \sqrt{2N+m-1}}
		= \frac{2}{\pi} \cdot \frac{1}{\sqrt{(m-1)(2N+m-1)}}.
	\end{align*}
	Substituting this prefactor back into the sum yields \eqref{eq:triple-theta}.
	
	For \eqref{eq:triple-theta-simplified}, we substitute the expression for $\alpha+\beta$ into the exponent and use $N=n-2m$ in the denominator, noting that $2N+m-1 = 2(n-2m)+m-1 = 2n-3m-1$.
\end{proof}

\medskip

\begin{lemma}[Local Limit Theorem and Central Binomial Mass]\label{lem:LLT}
	Let $r\ge2$, $c:=\lfloor r/2\rfloor$ and $p_r(t):=2^{-r}\binom{r}{t}$.
	Let $g_r(t):=\sqrt{\frac{2}{\pi r}}\exp\!\big(-(2t-r)^2/(2r)\big)$.
	There exists an absolute $C_0>0$ such that
	\begin{equation}\label{eq:LLT}
		\sup_{t\in\mathbb Z}\,|\,p_r(t)-g_r(t)\,|\ \le\ C_0\,r^{-3/2}.
	\end{equation}
	In particular, at the center $t=c$,
	\begin{equation}\label{eq:central-LLT}
		\Big|\,p_r(c)-g_r(c)\,\Big|
		\ \le\ C_0\,r^{-3/2},
		\qquad
		g_r(c)=
		\begin{cases}
			\sqrt{\frac{2}{\pi r}}, & r \text{ even},\\[1ex]
			\sqrt{\frac{2}{\pi r}}\,e^{-1/(2r)}, & r \text{ odd}.
		\end{cases}
	\end{equation}
	Hence, for all $r\ge2$,
	\begin{equation}\label{eq:central-LLT-bds}
		\sqrt{\frac{2}{\pi r}}\,e^{-1/(2r)}-C_0\,r^{-3/2}
		\ \le\ p_r(c)\ \le\ \sqrt{\frac{2}{\pi r}}+C_0\,r^{-3/2}.
	\end{equation}
\end{lemma}

\begin{proof}
	This is a classical uniform local limit theorem, see \cite[Chapter 7, Theorem 13]{petrov2012sums} (with $p=q=\tfrac12$). Evaluating at $t=c$ gives \eqref{eq:central-LLT}; the
	bounds \eqref{eq:central-LLT-bds} follow since $g_r(c)$ is as displayed.
\end{proof}

\medskip

\begin{lemma}[Moments of the squared-binomial weights]
\label{lem:squared-binomial-moments}
Let \(r\ge0\), and define
\[
    p_r(j):=2^{-r}\binom rj,
    \qquad
    S_r:=\sum_{j=0}^{r}p_r(j)^2,
    \qquad
    w_j:=\frac{p_r(j)^2}{S_r},
    \quad j=0,\ldots,r.
\]
Let \(J\) have law \(\mathbb P(J=j)=w_j\). Then
\[
    J\sim \operatorname{Hypergeom}(2r,r,r),
\]
with the convention that \(J=0\) deterministically when \(r=0\). Moreover,
\[
    \mathbb E[J]=\frac r2,
    \qquad
    \operatorname{Var}(J)
    =
    \frac{r^2}{4(2r-1)}
    =
    \frac r8+O(1),
\]
and, for every \(|\eta|\le 1/2\),
\[
    \mathbb E\left[
    \left(\frac r2-J+\eta\right)^2
    \right]
    =
    \frac r8+O(1),
    \qquad
    \mathbb E\left[
    \left(\frac r2-J+\eta\right)^4
    \right]
    =
    O\bigl((r+1)^2\bigr).
\]
\end{lemma}

\begin{proof}
By Vandermonde's identity,
\[
    S_r
    =
    2^{-2r}\sum_{j=0}^r\binom rj^2
    =
    2^{-2r}\binom{2r}{r}.
\]
Thus, for \(r\ge1\),
\[
    w_j
    =
    \frac{\binom rj^2}{\binom{2r}{r}}
    =
    \frac{\binom rj\binom r{r-j}}{\binom{2r}{r}},
    \qquad j=0,\ldots,r,
\]
which is the probability mass function of
\(\operatorname{Hypergeom}(2r,r,r)\). Hence
\[
    \mathbb E[J]=r\cdot \frac r{2r}=\frac r2.
\]
Using the standard variance formula for
\(X\sim\operatorname{Hypergeom}(N,K,n)\),
\[
    \operatorname{Var}(X)
    =
    n\frac K N\left(1-\frac K N\right)\frac{N-n}{N-1},
\]
with \(N=2r\), \(K=r\), and \(n=r\), gives
\[
    \operatorname{Var}(J)
    =
    r\cdot\frac12\cdot\frac12\cdot\frac r{2r-1}
    =
    \frac{r^2}{4(2r-1)}
    =
    \frac r8+O(1).
\]

It remains to bound the fourth moment. View \(J\) as the number of marked elements
in a sample of size \(r\) drawn without replacement from a population of \(2r\)
elements, exactly \(r\) of which are marked. By Hoeffding's inequality for sampling
without replacement,
\[
    \mathbb P\left(\left|J-\frac r2\right|\ge t\right)
    \le
    2\exp\left(-\frac{2t^2}{r}\right),
    \qquad t\ge0.
\]
Therefore, using the tail-integral identity for nonnegative random variables,
\[
\begin{aligned}
    \mathbb E\left[\left|J-\frac r2\right|^4\right]
    &=
    \int_0^\infty 4t^3
    \mathbb P\left(\left|J-\frac r2\right|\ge t\right)\,dt \\
    &\le
    8\int_0^\infty t^3
    \exp\left(-\frac{2t^2}{r}\right)\,dt
    =
    r^2 .
\end{aligned}
\]
Finally, since \(\mathbb E[J-r/2]=0\), for \(|\eta|\le1/2\),
\[
\begin{aligned}
    \mathbb E\left[
    \left(\frac r2-J+\eta\right)^2
    \right]
    &=
    \operatorname{Var}(J)+\eta^2
    =
    \frac r8+O(1),
\end{aligned}
\]
and
\[
    |a+b|^4\le 8|a|^4+8|b|^4
\]
gives
\[
\begin{aligned}
    \mathbb E\left[
    \left(\frac r2-J+\eta\right)^4
    \right]
    &\le
    8\mathbb E\left[\left|J-\frac r2\right|^4\right]
    +
    8|\eta|^4  \\
    &=
    O\bigl((r+1)^2\bigr).
\end{aligned}
\]
The case \(r=0\) is immediate.
\end{proof}

\bigskip

\subsection{Exact Combinatorial Form of the Fold-Covariance}\label{app:maj_exact-combinatorial}

\begin{lemma}[Simplification of Covariance Term]
	\label{lemma:cov_simplification_final}
	Let $X_m \sim \mathrm{Bin}(m, 1/2)$, $\theta=(n-m)/2$, and $a(y) = \lfloor \theta-y \rfloor$. Then, for every fixed \(y\),
	\[\Cov_{X_m}(X_m, \mathbf{1}_{X_m > \theta-y}) = \frac{m}{4}\mathbb P(X_{m-1} = a(y))\]
	where $X_{m-1} \sim \mathrm{Bin}(m-1, 1/2)$.
\end{lemma}
\begin{proof}
	Throughout, let \(C(m, y) := \Cov_{X_m}(X_m, \mathbf{1}_{X_m > \theta-y})\). Let $a(y) = \lfloor \theta-y \rfloor$. The event $X_m > \theta-y$ is equivalent to $X_m \ge \lfloor \theta-y \rfloor + 1 = a(y)+1$.
	The covariance is $C(m,y) = \mathbb{E}_{X_m}[X_m \mathbf{1}_{X_m \ge a(y)+1}] - \mathbb{E}_{X_m}[X_m] \cdot \mathbb P(X_m \ge a(y)+1)$.
	Since $X_m \sim \text{Bin}(m, 1/2)$, its expectation is $\mathbb{E}_{X_m}[X_m] = m/2$.
	The first term is $\mathbb{E}_{X_m}[X_m \mathbf{1}_{X_m \ge a(y)+1}] = \sum_{j=a(y)+1}^m j \binom{m}{j} (1/2)^m$. Using $j \binom{m}{j} = m \binom{m-1}{j-1}$:
	$$ \sum_{j=a(y)+1}^m m \binom{m-1}{j-1} (1/2)^m = \frac{m}{2} \sum_{j'=a(y)}^{m-1} \binom{m-1}{j'} (1/2)^{m-1} = \frac{m}{2} \mathbb P(X_{m-1} \ge a(y)) $$
	where $X_{m-1} \sim \text{Bin}(m-1, 1/2)$.
	So, $C(m,y) = \frac{m}{2} \mathbb P(X_{m-1} \ge a(y)) - \frac{m}{2} \mathbb P(X_m \ge a(y)+1)$.
	To simplify $\mathbb P(X_m \ge j+1)$, let $X_m = X_{m-1} + B$, where $B \sim \text{Bernoulli}(1/2)$ is independent of $X_{m-1}$.
	\begin{align*} \mathbb P(X_m \ge j+1) &= \mathbb P(X_{m-1} + B \ge j+1 | B=0)\mathbb P(B=0) + \mathbb P(X_{m-1} + B \ge j+1 | B=1)\mathbb P(B=1) \\ &= \frac{1}{2}\mathbb P(X_{m-1} \ge j+1) + \frac{1}{2}\mathbb P(X_{m-1} \ge j)\end{align*}
	Substituting this with $j=a(y)$:
	\begin{align*} C(m,y) &= \frac{m}{2} \left[\mathbb P(X_{m-1} \ge a(y)) - \left(\frac{1}{2}\mathbb P(X_{m-1} \ge a(y)+1) + \frac{1}{2}\mathbb P(X_{m-1} \ge a(y))\right)\right] \\ &= \frac{m}{4} \left[\mathbb P(X_{m-1} \ge a(y)) - \mathbb P(X_{m-1} \ge a(y)+1)\right] = \frac{m}{4} \mathbb P(X_{m-1} = a(y)) 
    \end{align*}
\end{proof}

\medskip

\begin{theorem}\label{thm:maj_exact-combinatorial}
It holds that
\[
\Cov(\widehat L_1, \widehat L_2)=2^{-n} \sum_{j=0}^{m-1} \binom{m-1}{j}^2 \binom{n-2m}{\lfloor (n-m)/2\rfloor-j }
\]
\end{theorem}
\begin{proof}
We know from combining Corollary~\ref{cor:majFactor} and Lemma~\ref{lemma:cov_simplification_final} that $\Cov(\widehat L_1, \widehat L_2)=\tfrac{1}{4} \mathbb E_Y[\mathbb P(X=\lfloor \theta-Y \rfloor \mid Y)^2]$ where \(X\sim \operatorname{Bin}(m-1,1/2)\) is independent of \(Y\). Furthermore,
\begin{align*} 
	 \mathbb E_Y[\mathbb P(X=&\lfloor \theta-Y \rfloor \mid Y)^2]\\
	&= \mathbb E_Y[\mathbb P(X_1=\lfloor \theta-Y \rfloor, X_2=\lfloor \theta-Y \rfloor | Y) ] \quad (\text{introducing } X_1, X_2 \text{ cond. indep. given } Y) \\
	&= \mathbb P(X_1=\lfloor \theta-Y \rfloor, X_2=\lfloor \theta-Y \rfloor) \quad (\text{by Law of Total Expectation}) \\
	&= \sum_{j=0}^{m-1} \mathbb P(X_1 = j, X_2=j,  \lfloor \theta-Y \rfloor=j) \quad (\text{summing over the support of } X_1, X_2)\\
	&= \sum_{j=0}^{m-1} \mathbb P(X_1=j, X_2=j) \cdot\mathbb P(j=\lfloor \theta-Y \rfloor) \quad (\text{by independence of } (X_1, X_2) \text{ and } Y)\\
	&= \sum_{j=0}^{m-1} \mathbb P(X_1=j)^2 \mathbb P(\lfloor \theta-Y \rfloor=j) \quad (\text{by independence of } X_1, X_2)\\
	&=\left(\frac{1}{2^{n-2}}\right) \sum_{j=0}^{m-1} \binom{m-1}{j}^2 \binom{n-2m}{\lfloor (n-m)/2\rfloor-j }  \quad (\text{writing out definition}).
\end{align*}
Multiplying by the prefactor \(1/4\) in the first display gives the claimed factor \(2^{-n}\).
\end{proof}

\bigskip\bigskip

\subsection{Main Results for the Majority Algorithm Fold-Covariance}

\begin{theorem}[Precise binomial form]\label{thm:maj_sublin}
	Fix integers $n$ and $m|n$.
Let $N:=n-2m$, and choose $N_C$ as the even integer closest to $n-\tfrac32 m$.

Then, one has
\[
\operatorname{Cov}(n,m)
=
S_{m-1}\frac{1}{2\sqrt{\pi(2n-3m)}}
+
O\!\left(n^{-3/2}+\frac{m^{3/2}}{n^{5/2}}\right),
\]
uniformly for all \(1\le m\le n/3\), where the error term is negligible once \(m=o(n)\).

\end{theorem}
\begin{proof}
		Define
	\[
	q_t(r):=2^{-t}\binom{t}{r}.
	\]
	Set
	\[
	p(j):=2^{-(m-1)}\binom{m-1}{j},\quad
	m_0:=\Big\lfloor \frac{n-m}{2}\Big\rfloor,\quad
	P_N(r):=q_N(r),\quad P_{N_c}:=q_{N_c}(N_c/2).
	\]
	Then
	\begin{equation}\label{eq:E-split}
		E(n, m)=\sum_{j=0}^{m-1} p(j)^2\,P_N(m_0-j)
		\ =\ S_{m-1}\,P_{N_c}\;+\;R_1,
		\quad
		R_1:=\sum_{j=0}^{m-1} p(j)^2\big(P_N(m_0-j)-P_{N_c}\big).
	\end{equation}
	
	\paragraph{Step 1: LLT expansions.}
	Apply \eqref{eq:LLT} to $p_{N}(m_0-j)$ and $S_{N_c/2}$.
	\[
	p_N(m_0-j)=G_N(j)+\delta_N(j),
	\qquad
	S_{N_c/2}=G_{N_c}+\delta_{N_c},
	\]
	where
	\[
	G_N(j):=\frac{1}{\sqrt{\pi N/2}}\exp\!\Big(-\frac{2\Delta_j^2}{N}\Big),\quad
	G_{N_c}:=\frac{1}{\sqrt{\pi N_c/2}},
	\quad
	\Delta_j:=m_0-j-\frac{N}{2}=\frac{m}{2}-j-\theta,\ \theta\in[0,1),
	\]
	and $|\delta_N(j)|\le C_{\LLT}N^{-3/2}$, $|\delta_{N_c}|\le C_{\LLT}N_c^{-3/2}$.
	
	Rigorously, $P_{N_c} := p_{N_c}(c_{N_c})$ with $c_{N_c} := \lfloor N_c/2 \rfloor$, so $G_{N_c} := g_{N_c}(c_{N_c}) = \sqrt{\frac{2}{\pi N_c}}e^{-(2c_{N_c}-N_c)^2/(2N_c)}$. By \eqref{eq:central-LLT}
	this is $\sqrt{\frac{2}{\pi N_c}}$ if $N_c$ even or $\sqrt{\frac{2}{\pi N_c}}e^{-1/(2N_c)}$ if $N_c$ odd. As $e^{-1/(2N_c)}=1+O(N_c^{-1})$, in both cases $G_{N_c} = \sqrt{\frac{2}{\pi N_c}} + O(N_c^{-3/2}).$
	\paragraph{Step 2: Decomposition of $R_1$.}
	Plugging in the Gaussian approximation, we get
	\begin{equation}\label{eq:R1-main}
		R_1=\sum_{j=0}^{m-1}p_{m-1}(j)^2\Big(G_N(j)-G_{N_c}\Big)
		\;+\;\underbrace{\sum_{j=0}^{m-1}p_{m-1}(j)^2\big(\delta_N(j)-\delta_{N_c}\big)}_{=:R_{\LLT}}.
	\end{equation}
	
	\paragraph{Step 3: Bounding the LLT remainder.}
	By $S_{m-1}\le 1$ and the local limit theorem bound of Lemma~\ref{lem:LLT},
	\[
	|R_{\LLT}|
	\ \le\ \sum_j p_{m-1}(j)^2\big(|\delta_N(j)|+|\delta_{N_c}|\big)
	\ \le\ C_{\LLT}\Big(\frac{1}{N^{3/2}}+\frac{1}{N_c^{3/2}}\Big)
	\ =\ O\!\Big(\frac{1}{n^{3/2}}\Big).
	\]
	
	\paragraph{Step 4. $N_c$ from the first–order optimal Gaussian central term.}

Next, we bound
\[
\sum_{j=0}^{m-1}p(j)^2\bigl(G_N(j)-G_{N_c}\bigr),
\]
where \(N=n-2m\). Define the probability weights
\[
    w_j:=\frac{p(j)^2}{S_{m-1}},
    \qquad j\in\{0,\ldots,m-1\},
\]
and let \(\mathbb E_w\) denote expectation with respect to these weights. Let
\(J\sim w\). By symmetry,
\[
    \mathbb E_w[J]=\frac{m-1}{2}.
\]
As before, write
\[
    \Delta_j:=m_0-j-\frac N2
    =
    \frac m2-j-\theta,
    \qquad \theta\in[0,1).
\]
Then
\[
    \Delta_J=\left(\frac{m-1}{2}-J\right)+\left(\frac12-\theta\right).
\]

By Lemma~\ref{lem:squared-binomial-moments}, applied with \(r=m-1\), we have
\[
    J\sim \operatorname{Hypergeom}\bigl(2(m-1),m-1,m-1\bigr),
\]
and
\[
    \mathbb E_w[J]=\frac{m-1}{2},
    \qquad
    \operatorname{Var}_w(J)
    =
    \frac{(m-1)^2}{4(2m-3)}
    =
    \frac{m-1}{8}+O(1).
\]
Since
\[
    \Delta_J
    =
    \left(\frac{m-1}{2}-J\right)
    +
    \left(\frac12-\theta\right),
    \qquad \theta\in[0,1),
\]
Lemma~\ref{lem:squared-binomial-moments}, with
\(\eta=1/2-\theta\), gives
\[
    \mathbb E_w[\Delta_J^2]
    =
    \frac{m-1}{8}+O(1),
    \qquad
    \mathbb E_w[\Delta_J^4]
    =
    O(m^2).
\]

Set
\[
    c(t):=\sqrt{\frac2\pi}\,t^{-1/2}.
\]
Then
\[
\begin{aligned}
\sum_{j=0}^{m-1}p(j)^2\bigl(G_N(j)-G_{N_c}\bigr)
&=
S_{m-1}
\left\{
c(N)\mathbb E_w\!\left[
\exp\!\left(-\frac{2\Delta_J^2}{N}\right)-1
\right]
-\bigl(c(N_c)-c(N)\bigr)
\right\}.
\end{aligned}
\]
Let
\[
    x_J:=\frac{2\Delta_J^2}{N}.
\]
Since
\[
    0\le e^{-x}-(1-x)\le \frac{x^2}{2},
    \qquad x\ge0,
\]
we get
\[
\begin{aligned}
\mathbb E_w\!\left[
\exp\!\left(-\frac{2\Delta_J^2}{N}\right)-1
\right]
&=
-\frac{2}{N}\mathbb E_w[\Delta_J^2]
+
O\!\left(\frac{\mathbb E_w[\Delta_J^4]}{N^2}\right) \\
&=
-\frac{m-1}{4N}
+
O\!\left(\frac1N+\frac{m^2}{N^2}\right).
\end{aligned}
\]
On the other hand, since \(N_c\) is the even integer closest to
\(n-\frac32m\), we have
\[
    N_c-N=\frac m2+O(1).
\]
A Taylor expansion of \(c(t)\) around \(N\) gives
\[
\begin{aligned}
    c(N_c)-c(N)
    &=
    -\frac{c(N)}{2}\frac{N_c-N}{N}
    +
    O\!\left(c(N)\frac{(N_c-N)^2}{N^2}\right) \\
    &=
    -c(N)\frac{m}{4N}
    +
    O\!\left(\frac{c(N)}{N}
    +
    c(N)\frac{m^2}{N^2}\right).
\end{aligned}
\]
Combining the last two displays,
\[
\left|
\sum_{j=0}^{m-1}p(j)^2\bigl(G_N(j)-G_{N_c}\bigr)
\right|
\le
C S_{m-1}
\left(
\frac{c(N)}{N}
+
c(N)\frac{m^2}{N^2}
\right).
\]
Since \(m\le n/3\), we have \(N\asymp n\) and \(c(N)\asymp n^{-1/2}\).
Also \(S_{m-1}=O(m^{-1/2})\), with the convention that this bound is harmless
at \(m=1\). Hence
\[
\left|
\sum_{j=0}^{m-1}p(j)^2\bigl(G_N(j)-G_{N_c}\bigr)
\right|
=
O\!\left(n^{-3/2}+\frac{m^{3/2}}{n^{5/2}}\right).
\]

	\paragraph{Step 5: Completing the proof}
Collecting Step 3 and Step 4, we have
\[
    |R_1|
    \le
    C\left(n^{-3/2}+\frac{m^{3/2}}{n^{5/2}}\right),
\]
uniformly for \(1\le m\le n/3\). Furthermore,
\[
    P_{N_c}
    =
    G_{N_c}+\delta_{N_c}
    =
    \sqrt{\frac{2}{\pi N_c}}
    +
    O(N_c^{-3/2}).
\]
Since
\[
    N_c=\frac{2n-3m}{2}+O(1),
\]
we also have
\[
    \sqrt{\frac{2}{\pi N_c}}
    =
    \frac{2}{\sqrt{\pi(2n-3m)}}+O(n^{-3/2}).
\]
Thus, 
\[
\begin{aligned}
    E(n,m)
    &=
    S_{m-1}\frac{2}{\sqrt{\pi(2n-3m)}}
    +
    O\!\left(n^{-3/2}+\frac{m^{3/2}}{n^{5/2}}\right).
\end{aligned}
\]
Recalling that \(\operatorname{Cov}(n,m)=E(n,m)/4\), we obtain
\[
    \operatorname{Cov}(n,m)
    =
    S_{m-1}\frac{1}{2\sqrt{\pi(2n-3m)}}
    +
    O\!\left(n^{-3/2}+\frac{m^{3/2}}{n^{5/2}}\right),
\]
uniformly for \(1\le m\le n/3\). This completes the proof.

\end{proof}

\medskip

\begin{theorem}[Endpoint cases]\label{thm:maj_small-k-mono}

	\noindent{\bf (A) Exact expression for $m=1$.}
It holds that
	\[
	\Cov(n, 1)=2^{-n}\binom{\,n-2\,}{\big\lfloor \frac{n-1}{2}\big\rfloor}.
	\]
	Consequently, by the LLT in Lemma~\ref{lem:LLT},
	\[
	\Cov(n, 1)=\sqrt{\frac{1}{8\pi(n-2)}}\;+\;O\!\Big(\frac{1}{n^{3/2}}\Big)
	=\sqrt{\frac{1}{8\pi n}}\;+\;O\!\Big(\frac{1}{n}\Big).
	\]

	\medskip
	\noindent{\bf (B) Endpoint asymptotic for $m=n/2$.}
	It holds that
		\[
	\Cov\Big(n,\tfrac{n}{2}\Big)=\frac{1}{\pi(n-2)}+O\!\Big(\frac{1}{n^{2}}\Big)
	=\frac{1}{\pi n}+O\Big(\frac{1}{n^{2}}\Big).
	\]
\end{theorem}

\begin{proof}

	\textbf{(A) The case $m=1$.}
	When $m=1$, the sum has only $j=0$ and $\binom{m-1}{0}^2=1$, so
	\[
	E(n, 1)=2^{-(n-2)}\binom{n-2}{\big\lfloor \frac{n-1}{2}\big\rfloor}.
	\]
	This is exactly the central (or near-central) mass of $\mathrm{Bin}(n-2,\tfrac12)$; by Lemma~\ref{lem:LLT},
	\[
	E(n, 1)=\frac{1}{\sqrt{\pi (n-2)/2}}+O\!\Big(\frac{1}{(n-2)^{3/2}}\Big)
	=\sqrt{\frac{2}{\pi(n-2)}}+O\!\Big(\frac{1}{n^{3/2}}\Big)
	=\sqrt{\frac{2}{\pi n}}+O\!\Big(\frac{1}{n}\Big).
	\]

	\textbf{(B) The case $m=\frac{n}{2}$ (so $n$ even).}
	Set $\ell:=\frac{n}{2}-1$ and observe
	\[
	E\Big(n,\tfrac{n}{2}\Big)
	=2^{-(n-2)}\sum_{j=0}^{\ell}\binom{\ell}{j}^{\!2}\binom{0}{\lfloor n/4\rfloor-j}.
	\]
	Since $\binom{0}{r}=\mathbf{1}\{r=0\}$, only the term $j=r:=\lfloor n/4\rfloor$ survives:
	\[
	E\Big(n,\tfrac{n}{2}\Big)
	=\Big(2^{-\ell}\binom{\ell}{r}\Big)^{2}=:q_\ell(r)^{2},
	\]
	i.e. it is the square of a symmetric binomial mass $q_\ell(r)=\Pr\{\mathrm{Bin}(\ell,\tfrac12)=r\}$.

    Note that
\[
    r-\frac{\ell}{2}
    =
    \begin{cases}
        1/2, & n\equiv 0 \pmod 4,\\
        0,   & n\equiv 2 \pmod 4.
    \end{cases}
\]
In either case,
\[
    \frac{2(r-\ell/2)^2}{\ell}=O(\ell^{-1}),
\]
and hence
\[
    q_\ell(r)
    =
    \sqrt{\frac{2}{\pi\ell}}\left(1+O(\ell^{-1})\right).
\]
Therefore
\[
    E(n,n/2)=q_\ell(r)^2
    =
    \frac{2}{\pi\ell}+O(\ell^{-2})
    =
    \frac{4}{\pi(n-2)}+O(n^{-2}),
\]
and since \(\operatorname{Cov}(n,n/2)=E(n,n/2)/4\),
\[
    \operatorname{Cov}(n,n/2)
    =
    \frac{1}{\pi(n-2)}+O(n^{-2}).
\]
	
\end{proof}

		\medskip

\begin{theorem}[Covariance for large $m$]\label{thm:maj_large-m}
		Uniformly over all integer sequences \(m=m_n\) satisfying
\(m=\omega_n(1)\) and \(m\le n/3\),
		\[
		\Cov(n,m)=\frac{1}{2\pi\sqrt{(m-1)(2n-3m)}}
		 +O\bigg(\frac{1}{\sqrt{n}m^{3/2}}\bigg).
		\]
	\end{theorem}
	
	\begin{proof}
		Write $p_r(t):=2^{-r}\binom{r}{t}$ and
		$g_r(t):=\sqrt{\frac{2}{\pi r}}\exp\!\big(-(2t-r)^2/(2r)\big)$.
		Set
		\[
		E(n,m)=\sum_{j=0}^{m-1} p_{m-1}(j)^2\,p_N\!\Big(\tfrac N2+\Delta_j\Big),
		\qquad
		\Delta_j:=\ell-j-\tfrac N2.
		\]

Let \(N:=n-2m\), \(\ell:=\lfloor (n-m)/2\rfloor\), and write
\[
    \varepsilon:=\frac{n-m}{2}-\ell\in\{0,1/2\}.
\]
        
		We decompose
		\begin{align*}
		E(n,m)-\frac{2}{\pi\sqrt{(m-1)(2n-3m)}}
		&=
        \underbrace{\big(E(n,m)-\sum_J p_{m-1}^2 g_N\big)}_{T_1}
		+\underbrace{\big(\sum_J p_{m-1}^2 g_N-\sum_J g_{m-1}^2 g_N\big)}_{T_2}\\
        +\underbrace{\big(\sum_J g_{m-1}^2 g_N-\sum_\Z g_{m-1}^2 g_N\big)}_{T_3}&
		+\underbrace{\big(\sum_\Z g_{m-1}^2 g_N-\frac{2}{\pi\sqrt{(m-1)(2n-3m)}}\big)}_{T_4},
		\end{align*}
		where $J=\{0,\dots,m-1\}$ in $T_1,T_2$.
		
		\emph{1) $T_1$: replace $p_N$ by $g_N$ (uniform LLT).}
		The LLT Lemma~\ref{lem:LLT} yields $\sup_t|p_N(t)-g_N(t)|\le C_0 N^{-3/2}$, hence
		\[
		|T_1|\ \le\ C_0 N^{-3/2}\sum_{j=0}^{m-1}p_{m-1}(j)^2
		= C_0 N^{-3/2}\,p_{2m-2}(m-1).
		\]
		By the same LLT at $r=2m-2$, $p_{2m-2}(m-1)\le g_{2m-2}(m-1)+C_0(2m-2)^{-3/2}
		\le \frac{1}{\sqrt{\pi(m-1)}}+\frac{C_0}{2^{3/2}(m-1)^{3/2}}$.
		Since $N\ge n/3$,
		\[
		|T_1| \ \le\ \frac{C}{\sqrt{m}n^{3/2}}.
		\]
		
		\emph{2) $T_2$: replace $p_{m-1}$ by $g_{m-1}$.}
		Let $\delta_j:=p_{m-1}(j)-g_{m-1}(j)$. Then
		\[
		|T_2|
		= \Big|\sum_{j=0}^{m-1}\delta_j\,(p_{m-1}(j)+g_{m-1}(j))\,g_N\!\Big(\tfrac N2+\Delta_j\Big)\Big|
		\le \big(\sup_j|\delta_j|\big)\,\Sigma,
		\]
		with
		\[
		\Sigma:=\sum_{j=0}^{m-1}(p_{m-1}(j)+g_{m-1}(j))\,g_N\!\Big(\tfrac N2+\Delta_j\Big).
		\]
		Uniform LLT (Lemma~\ref{lem:LLT}) at scale $m-1$ gives $\sup_j|\delta_j|\le C_0 (m-1)^{-3/2}$.
		Moreover, upper bounding the average by the maximum,
		\[
		\sum_{j=0}^{m-1}p_{m-1}(j)\,g_N\!\Big(\tfrac N2+\Delta_j\Big)
		\le \sup_t g_N(t)\ \le\ \sqrt{\frac{2}{\pi N}},
		\]
		and similarly, by extending to $\mathbb Z$ and using the lattice Gaussian–Gaussian convolution,
		\[
		\sum_{j=0}^{m-1}g_{m-1}(j)\,g_N\!\Big(\tfrac N2+\Delta_j\Big)
		\le \sum_{j\in\mathbb Z} g_{m-1}(j)\,g_N(\ell-j)\ \le\ \frac{C'}{\sqrt{N}}.
		\]
		Therefore $\Sigma\le C/\sqrt{N}$, and
		\[
		|T_2|\ \le\ \frac{C_0}{(m-1)^{3/2}}\cdot \frac{C}{\sqrt{N}}
		\ \le\ \frac{C}{m^{3/2}\sqrt{n}}.
		\]

\emph{3) $T_3$: truncating the full-lattice Gaussian sum.}

Here
\[
T_3
=
\sum_{j=0}^{m-1} g_{m-1}(j)^2 g_N(\ell-j)
-
\sum_{j\in\mathbb Z} g_{m-1}(j)^2 g_N(\ell-j),
\]
where \(N=n-2m\) and \(\ell=\lfloor (n-m)/2\rfloor\). Put
\[
    \varepsilon:=\frac{n-m}{2}-\ell\in\{0,1/2\},
    \qquad
    a:=\frac{m-1}{2},
    \qquad
    b:=\frac m2-\varepsilon,
\]
and
\[
    \alpha:=\frac4{m-1},
    \qquad
    \beta:=\frac2N,
    \qquad
    \gamma:=\alpha+\beta.
\]
Then
\[
g_{m-1}(j)^2g_N(\ell-j)
=
C_{m,N}
\exp\{-\alpha(j-a)^2-\beta(j-b)^2\}
=
C'_{m,N}\exp\{-\gamma(j-\mu)^2\},
\]
where
\[
    \mu
    =
    \frac{\alpha a+\beta b}{\gamma}
    =
    \frac{(m-1)(2N+m-2\varepsilon)}{2(2N+m-1)}
\]
and
\[
    \sigma^2:=\frac1{2\gamma}
    =
    \frac{(m-1)N}{4(2N+m-1)}.
\]
Since \(m\le n/3\), we have \(N\ge m\). Hence, for all sufficiently large
\(m\),
\[
    \mu\ge c m,
    \qquad
    m-1-\mu\ge c m,
    \qquad
    \sigma^2\le C m
\]
for universal constants \(c,C>0\). Therefore both tails outside
\(\{0,\ldots,m-1\}\) are Gaussian tails at distance at least \(c\sqrt m\)
standard deviations. Consequently,
\[
    |T_3|
    \le
    C\frac{e^{-c m}}{\sqrt{mN}}
    \le
    C\frac{1}{\sqrt n\,m^{3/2}},
\]
where the last inequality uses \(N\asymp n\) for \(m\le n/3\) and
\(\sup_{x\ge1}xe^{-cx}<\infty\).

\emph{4) $T_4$: Gaussian triple product to the simple main term.}
		The full-lattice Gaussian sum has the exact form
		\[
		\sum_{j\in\mathbb Z} g_{m-1}(j)^2\,g_N(\ell-j)
		=\frac{2}{\pi\sqrt{(m-1)(2n-3m)}}\cdot A_{n,m}\,\Theta_{n,m},
		\]
		with
		\[
		A_{n,m}=\frac{\sqrt{2n-3m}}{\sqrt{2n-3m-1}}\,
		e^{-1/(2n-3m-1)}=1+ O(1/n),
		\]
		and
		\[
		\Theta_{n,m}=1+2\sum_{t\ge1}\exp\!\Big(-\frac{\pi^2 t^2 (m-1)N}{2(2n-3m-1)}\Big)\cos(2\pi t\mu).
		\]

When \(\varepsilon=0\), this is exactly Lemma~\ref{lem:triple-sum-theta}. When
\(\varepsilon=1/2\), the same calculation applies with the second Gaussian center
shifted from \(m/2\) to \(m/2-\varepsilon\). This changes the constant factor
\[
    \exp\!\left\{-\frac{\alpha\beta}{\alpha+\beta}(a-b)^2\right\}
\]
and the phase \(\mu\), but since \(|a-b|\le 1/2\) and
\[
    \frac{\alpha\beta}{\alpha+\beta}
    =
    O(1/n),
\]
the prefactor remains \(1+O(1/n)\). The bound on the theta correction is unchanged,
as it uses only absolute values.

Let $A = \frac{\pi^2 (m-1)N}{2(2n-3m-1)}$. Using the triangle inequality and $|\cos(\cdot)| \le 1$, we can bound the error term:
\begin{align*}
	|\Theta_{n,m}-1| &\le 2\sum_{t\ge1}\exp(-A t^2)
\end{align*}
Since $t^2 \ge t$ for $t \ge 1$, we can further bound this by a geometric series:
\begin{align*}
	|\Theta_{n,m}-1| &\le 2\sum_{t\ge1}\exp(-A t) = 2 \frac{\exp(-A)}{1-\exp(-A)}.
\end{align*}

Since \(m\le n/3\), we have \(N\asymp n\), and
\[
    A
    =
    \frac{\pi^2(m-1)N}{2(2n-3m-1)}
    =
    \Theta(m).
\]
Hence
\(
    |\Theta_{n,m}-1|\le C e^{-cm}.
\)
Consequently,
\(
    |A_{n,m}\Theta_{n,m}-1|
    \le C(n^{-1}+e^{-cm}),
\)
and therefore
\[
\begin{aligned}
|T_4|
&\le
\frac{C}{\sqrt{(m-1)(2n-3m)}}
\left(n^{-1}+e^{-cm}\right) \\
&\le
C\left(
\frac{1}{\sqrt m\,n^{3/2}}
+
\frac{1}{\sqrt n\,m^{3/2}}
\right),
\end{aligned}
\]
where the last inequality uses \(2n-3m\asymp n\) and \(e^{-cm}\le C/m\).

		\emph{5) Conclusion.}
		Adding the four bounds gives
		\(
		|T_1|+|T_2|+|T_3|+|T_4|= O(\frac{1}{\sqrt{m}n^{3/2}})+O(\frac{1}{\sqrt{n}m^{3/2}}),
		\)
		uniformly for \(m=\omega_n(1)\) and \(m\le n/3\). The first error term is asymptotically dominated by the second error term and hence absorbed. The result is obtained by recalling $\Cov(n,m)=E(n,m)/4$.
	\end{proof}

	\medskip
		
	\begin{theorem}[Monotonicity and minimizer]\label{thm:maj_mono-minimizer}
Let
\[
    \mathcal M_n:=\{m\in \mathbb N: m\mid n,\ 1\le m\le n/2\}.
\]
For all sufficiently large \(n\), the covariance is strictly decreasing over admissible
fold sizes up to \(n/3\): if \(m_1,m_2\in \mathcal M_n\) and
\[
    m_1<m_2\le n/3,
\]
then
\[
    \operatorname{Cov}(n,m_1)>\operatorname{Cov}(n,m_2).
\]
Moreover, along subsequences with \(6\mid n\),
\[
    \operatorname{Cov}(n,n/3)<\operatorname{Cov}(n,n/2).
\]
Consequently, along subsequences with \(3\mid n\), the fold covariance is minimized at \(k=3\), among all admissible equal-fold choices.
\end{theorem}

\begin{proof}
We write \(a_n\lesssim b_n\) if \(a_n\le Cb_n\) for a universal constant \(C\).
We first prove the monotonicity for admissible \(m_1<m_2\le n/3\).

Suppose first that \(m_2=o(n)\). We use the combinatorial estimate from
Theorem~\ref{thm:maj_sublin}, which is precise in the sublinear regime. Define
\[
    S_{m-1}:=2^{-2m+2}\binom{2m-2}{m-1},
    \qquad
    \Phi(n,m):=
    S_{m-1}\frac{1}{2\sqrt{\pi(2n-3m)}}.
\]
Then, uniformly for \(1\le m\le n/3\),
\begin{equation}\label{eq:A.11}
    \operatorname{Cov}(n,m)
    =
    \Phi(n,m)
    +
    O\!\left(n^{-3/2}+m^{3/2}n^{-5/2}\right).
\end{equation}

Let
\[
    d:=m_2-m_1>0.
\]
Since \(m_1,m_2\mid n\), we have
\[
    \operatorname{lcm}(m_1,m_2)\mid n.
\]
Moreover, \(\gcd(m_1,m_2)\mid d\), hence \(\gcd(m_1,m_2)\le d\). Therefore
\[
    n\ge \operatorname{lcm}(m_1,m_2)
    =
    \frac{m_1m_2}{\gcd(m_1,m_2)}
    \ge
    \frac{m_1m_2}{d},
\]
and so
\(
    d\ge \frac{m_1m_2}{n}\ge \frac{m_1^2}{n}.
\)
Of course also \(d\ge 1\). Consequently,
\begin{equation}\label{eq:A.13}
    \frac d{m_1}\ge \max\left\{\frac1{m_1},\frac{m_1}{n}\right\}.
\end{equation}

We now compare the leading terms \(\Phi(n,m_1)\) and \(\Phi(n,m_2)\). Using the exact ratio
\[
    \frac{S_t}{S_{t-1}}=\frac{2t-1}{2t}=1-\frac1{2t},
\]
we obtain
\[
    \frac{S_{m_2-1}}{S_{m_1-1}}
    =
    \prod_{r=m_1}^{m_2-1}\left(1-\frac1{2r}\right)
    \le
    \exp\left(-\frac12\sum_{r=m_1}^{m_2-1}\frac1r\right).
\]
Hence, for \(m_2=o(n)\),
\[
\begin{aligned}
    \log\frac{\Phi(n,m_2)}{\Phi(n,m_1)}
    &=
    \log\frac{S_{m_2-1}}{S_{m_1-1}}
    +
    \frac12\log\frac{2n-3m_1}{2n-3m_2}                                      \\
    &\le
    -c_1\min\left\{1,\frac d{m_1}\right\}
    +
    C_1\frac d n                                                         \\
    &\le
    -c_2\min\left\{1,\frac d{m_1}\right\},
\end{aligned}
\]
for all sufficiently large \(n\), since \(m_2=o(n)\). Therefore
\begin{equation}\label{eq:A.14}
    \Phi(n,m_1)-\Phi(n,m_2)
    \ge
    c_3\,\Phi(n,m_1)\min\left\{1,\frac d{m_1}\right\}.
\end{equation}

It remains to check that the error in \eqref{eq:A.11} is negligible relative to the
right-hand side of \eqref{eq:A.14}. Since
\[
    \Phi(n,m_1)\asymp \frac{1}{\sqrt{m_1 n}},
\]
and since
\(
    \frac d{m_1}
    \ge
    \max\left\{\frac1{m_1},\frac{m_1}{n}\right\}
\) due to \eqref{eq:A.13},
we have, in the case \(d\le m_1\),
\[
\frac{n^{-3/2}}
     {\Phi(n,m_1)(d/m_1)}
\lesssim
\frac{\sqrt{m_1}/n}
     {\max\{1/m_1,m_1/n\}}
=o(1).
\]
Also \(d\le m_1\) implies \(m_2\le 2m_1\), and therefore
\[
\frac{m_2^{3/2}n^{-5/2}}
     {\Phi(n,m_1)(d/m_1)}
\lesssim
\frac{m_1^2/n^2}
     {\max\{1/m_1,m_1/n\}}
=o(1).
\]
If \(d>m_1\), then \(\min\{1,d/m_1\}=1\), and
\[
    \frac{n^{-3/2}}{\Phi(n,m_1)}
    \lesssim
    \frac{\sqrt{m_1}}{n}
    =
    o(1),
\]
while, using \(m_1\le m_2=o(n)\),
\[
    \frac{m_2^{3/2}n^{-5/2}}{\Phi(n,m_1)}
    \lesssim
    \frac{m_2^{3/2}\sqrt{m_1}}{n^2}
    \le
    \left(\frac{m_2}{n}\right)^2
    =
    o(1).
\]
Combining these bounds with \eqref{eq:A.11} and \eqref{eq:A.14} gives
\[
    \operatorname{Cov}(n,m_1)>\operatorname{Cov}(n,m_2)
\]
throughout the sublinear regime \(m_2=o(n)\).

Next suppose \(m_1=o(n)\) but \(m_2\) is not \(o(n)\). Then the preceding estimates give
\[
    \operatorname{Cov}(n,m_1)\asymp \frac1{\sqrt{m_1n}},
\]
whereas Theorem~\ref{thm:maj_large-m} gives \(\operatorname{Cov}(n,m_2)=O(n^{-1})\). Since
\(m_1=o(n)\), we have \(1/\sqrt{m_1n}\gg 1/n\), and the desired inequality follows.

It remains only to consider the regime in which \(m_1\) is of order \(n\). Write
\[
    m_i=\frac n{k_i},\qquad i=1,2.
\]
Since \(m_1<m_2\le n/3\), we have integers \(k_1>k_2\ge 3\). Along any such
linear-regime subsequence, the possible \(k_i\)'s are bounded, so after passing to a
subsequence we may regard \(k_1,k_2\) as fixed. Theorem~\ref{thm:maj_large-m} yields
\[
    \operatorname{Cov}\!\left(n,\frac n k\right)
    =
    \frac{k}{2\pi n\sqrt{2k-3}}
    +
    O(n^{-2})
\]
for every fixed integer \(k\ge 3\). The function
\[
    h(k):=\frac{k}{\sqrt{2k-3}}
\]
is strictly increasing on the integers \(k\ge 3\): indeed \(h(4)>h(3)\), and for
real \(k>3\),
\[
    h'(k)=\frac{k-3}{(2k-3)^{3/2}}>0.
\]
Thus \(k_1>k_2\) implies
\[
    \operatorname{Cov}\!\left(n,\frac n{k_1}\right)
    >
    \operatorname{Cov}\!\left(n,\frac n{k_2}\right)
\]
for all sufficiently large \(n\). This proves the claimed monotonicity over
admissible \(m\le n/3\).

Finally assume \(6\mid n\). Then Theorem~\ref{thm:maj_large-m} at \(m=n/3\) and Theorem~\ref{thm:maj_small-k-mono}(B) at
\(m=n/2\) give

\[
    \operatorname{Cov}(n,n/3)
    =
    \frac{\sqrt 3}{2\pi n}
    +
    O(n^{-2}),
\]
and
\[
    \operatorname{Cov}(n,n/2)
    =
    \frac1{\pi n}
    +
    O(n^{-2}).
\]
Since \(\sqrt3/2<1\), it follows that
\[
    \operatorname{Cov}(n,n/3)<\operatorname{Cov}(n,n/2)
\]
for all sufficiently large \(n\) with \(6\mid n\).

If \(3\mid n\), then \(m=n/3\) is admissible. The preceding monotonicity shows that
it minimizes the covariance among all admissible \(m\le n/3\). The only admissible
fold size in \((n/3,n/2]\), if it exists, is \(m=n/2\), corresponding to \(k=2\);
when \(6\mid n\) the comparison above rules it out, and when \(n\) is odd it is not
admissible. Hence \(m=n/3\), equivalently \(k=3\), is the admissible covariance
minimizer along \(3\mid n\).
\end{proof}

\newpage

\section{Proof of Lemma~\ref{lem:anticorr}}\label{proof:anti}
\begin{proof}
Consider the following setup. Let $\mathcal X=[0,1]$, $\mathcal Y=\{0,1\}$, with input distribution $\mathcal D$ over $\mathcal{X}\times\mathcal{Y}$, where the marginal over $\mathcal{X}$ is uniform and $y=f(x)$ where 
 $f = \ind{x > 1/2}$. Consider the algorithm that outputs $\mathcal A(S^n)=\ind{1/2-p/2< ~\cdot~< 1-p/2}$ with $p=p(S)=\sum_i y_i/n$, and $\mathcal A(S^{n-m})=h_0$, where $h_0$ is the constant zero hypothesis. Then, $\widehat L^{(k)}_{\rm CV}=\sum_i \widehat L_i^{(k)}/k=\sum_i y_i/n=p(S)$ and $L(A(S^n))=p(S)$, so that the MSE is zero. For the loss-stability term, note that \(L(A(S^{n-m}))=1/2\). Hence, for \(n=2\) and \(m=1\),
\[
\mathbb E\left[
\left|L(A(S^n))-L(A(S^{n-m}))\right|
\right]
=
\mathbb E_{Y\sim\operatorname{Bin}(2,1/2)}
\left|
\frac{Y}{2}-\frac12
\right|
=
\frac14.
\]

\end{proof}

\section{Error in Theorem 5.3 of \citet{kearns}}\label{app:err}

Let us first recall their notion of stability in our notation. We say that a deterministic algorithm $\mathcal{A}$ has error stability $(\beta_1, \beta_2)$ if $\mathbb P_{S^{n-1}, (x,y)}[|L(\mathcal{A}( S^n)) - L(\mathcal{A}( S^{n-1}))| \ge \beta_2] \le \beta_1$ where $S^n = S^{n-1} \circ {(x,y)}$, and both $\beta_1$ and $\beta_2$ may be functions of $n$.

Consider the proof of their Theorem~5.3. There, they define the random variable $\chi(S^n) = \widehat L_{\text{CV}}^{(n)} - L(\mathcal{A}( S^n))$ and assume without loss of generality that with probability at least $\beta_1/2$, $L(\mathcal{A} (S^{n-1})) - L(\mathcal{A}( S^n)) \ge \beta_2$.

Next, their Lemma 4.1 correctly asserts that the expected cross-validation estimate equals the expected estimate of a single hold-out set, i.e., $\mathbb{E}_{S^n}[\chi(S^n)]=\mathbb{E}_{S^n}[L(\mathcal{A} (S^{n-1})) - L(\mathcal{A}( S^n))]$. Combined with the fact that with probability at least $\beta_1/2$, $L(\mathcal{A}( S^{n-1})) - L(\mathcal{A}( S^n) )\ge \beta_2$, the authors then claim that $\mathbf{E}_{S^n}[\chi(S^n)] \ge \frac{\beta_1}{2} \cdot \beta_2.$

This is incorrect, since $L(\mathcal{A}( S^{n-1})) - L(\mathcal{A}( S^n)) \ge \beta_2$ on part of the sample space does not rule out that this quantity can also be negative at other times. To illustrate this, let us consider an extreme case where $\beta_1=\beta_2=1$ by assuming that $\mathbb P(\{L(\mathcal{A}( S^{n-1})) - L(\mathcal{A}(S^n))=1\})=\beta_1/2=1/2$. This assumption does not rule out the possibility that $\mathbb P(\{L(\mathcal{A}( S^{n-1})) - L(\mathcal{A}( S^n))=-1\})=1/2$. In that case, $\mathbb{E}_{S^n}[\chi(S^n)]=0$, violating the alleged lower bound $\frac{\beta_1}{2} \cdot \beta_2=1/2$.

This directly contradicts our Lemma~\ref{lem:anticorr} because non-zero loss stability according to our Definition~\ref{def:lossstab} implies a lower bound on their high-probability error stability parameters, yet we prove in Lemma~\ref{lem:anticorr} that one can have non-zero loss stability and simultaneously zero MSE (which necessitates $\mathbf{E}_{S^n}[\chi(S^n)]=0$).

\section{Error in Theorem~2 of \citet{kale2011cross}}\label{app:errKale}
The key ingredient for deriving the main result of \citet[Theorem 2]{kale2011cross} is to obtain an upper bound on $\Cov_{S^n}(\widehat{L}_1^{(k)} - L_1^{(k)}, \widehat{L}_2^{(k)} - L_2^{(k)})$ (in their notation $\Cov_U(\text{gen}_1, \text{gen}_2)$) that scales linearly with a parameter measuring a certain notion of algorithmic stability (\emph{``mean square stability''}). To do so, the supposed identity $\E_{S_2}[\widehat{L}_1^{(k)}-L_1^{(k)} \mid S_1, S_3,\dots, S_k] = 0$ (in their notation $\E_{T'}[\text{gen}_1 \mid S, T] = 0$) is used twice. Define $S' := S^n \setminus S_2$ and $S'' := S^n \setminus (S_1 \circ S_2)$. One can see that
$$
\begin{aligned}
\E_{S_2}[\widehat{L}_1^{(k)} - L_1^{(k)} \mid S_1, S_3,\dots, S_k] &= \E_{S_2}\left[\frac{1}{m} \sum_{z' \in S_1} \ell(\mathcal{A}(S^n_{-1}), z') - \E_{z}[\ell(\mathcal{A}(S^n_{-1}), z)] \;\Big|\; S'\right] \\
&= \E_{S_2}\left[\frac{1}{m} \sum_{z' \in S_1} \ell(\mathcal{A}(S^n_{-1}), z') \;\Big|\; S'\right] - \E_{S_2, z}[\ell(\mathcal{A}(S^n_{-1}), z) \mid S'']
\end{aligned}
$$
where the two terms in the last line are functions of $S'$ and $S''$, respectively, and their difference is non-zero in general.

\end{document}